\documentclass{article}

\usepackage[final]{nips_2016}

\usepackage{graphicx}
\usepackage{subcaption}

\usepackage{pbox}

\usepackage[utf8]{inputenc} 
\usepackage[T1]{fontenc}    
\usepackage{hyperref}       
\usepackage{url}            
\usepackage{booktabs}       
\usepackage{amsfonts}       
\usepackage{nicefrac}       
\usepackage{microtype}      
\usepackage{amsmath, amsfonts, amsthm, amssymb}

\newtheorem{theorem}{Theorem}

\newtheorem{corollary}[theorem]{Corollary}

\renewenvironment{proof}{{\bf Proof:}}{\hfill\rule{2mm}{2mm}}

\newcommand{\inv}{^{-1}}                            
\newcommand{\sminus}{\backslash}                    
\newcommand{\R}{\mathbb{R}}                         
\newcommand{\Z}{\mathbb{Z}}                         
\newcommand{\e}{\varepsilon}                        
\newcommand{\X}{\mathcal{X}}                        
\newcommand{\E}{\mathop{\mathbb{E}}}                
\newcommand{\F}{\mathcal{F}}                        
\newcommand{\Var}{\mathbb{V}}                       
\newcommand{\Cov}{\operatornamewithlimits{Cov}}     
\renewcommand{\phi}{\varphi} 

\renewcommand{\tilde}{\widetilde}

\title{Efficient Nonparametric Smoothness Estimation} 

\author{
  Shashank Singh \\
  Statistics \& Machine Learning Departments \\
  Carnegie Mellon University \\
  Pittsburgh, PA 15213 \\
  \texttt{sss1@andrew.cmu.edu} \\
  \And
  Simon S. Du \\
  Machine Learning Departments \\
  Carnegie Mellon University \\
  Pittsburgh, PA 15213 \\
  \texttt{ssdu@cs.cmu.edu} \\
  \And
  Barnab\'as P\'oczos \\
  Machine Learning Departments \\
  Carnegie Mellon University \\
  Pittsburgh, PA 15213 \\
  \texttt{bapoczos@cs.cmu.edu} \\
}

\begin{document}

\maketitle

\begin{abstract}
Sobolev quantities (norms, inner products, and distances) of probability
density functions are important in the theory of nonparametric statistics, but
have rarely been used in practice, partly due to a lack of practical estimators.
They also include, as special cases, $L^2$ quantities which are used in many
applications. We propose and analyze a family of estimators for Sobolev
quantities of unknown probability density functions. We bound the bias and
variance of our estimators over finite samples, finding that they are generally
minimax rate-optimal. Our estimators are significantly more computationally
tractable than previous estimators, and exhibit a statistical/computational
trade-off allowing them to adapt to computational constraints. We also draw
theoretical connections to recent work on fast two-sample testing. Finally, we
empirically validate our estimators on synthetic data.
\end{abstract}

\section{Introduction}
\label{sec:introduction}
$L^2$ quantities (i.e., inner products, norms, and distances) of continuous
probability density functions are important information theoretic quantities
with many applications in machine learning and signal processing. For example,
estimates of the $L^2$ norm as can be used for
goodness-of-fit testing \citet{goria05new},
image registration and texture classification \citep{hero2002aes},
and parameter estimation in semi-parametric models \citet{Wolsztynski85minimum}.
$L^2$ inner products estimates can be used with linear or polynomial kernels to
generalize kernel methods to inputs which are distributions rather than
numerical vectors. \citep{poczos12kernelsForImages} Estimators of $L^2$
distance have been used for two-sample testing
\citep{anderson94L_2TwoSampleTest, pardo05divergenceMeasures}, transduction
learning \citep{quadrianto09distributionMatching}, and machine learning on
distributional inputs \citep{poczos12divergenceApplication}.
\citet{principe10RenyiEntropy} gives further applications of $L^2$ quantities to
adaptive information filtering, classification, and clustering.

$L^2$ quantities are a special case of less-well-known \emph{Sobolev
quantities}. Sobolev norms measure \emph{global smoothness} of a function in
terms of integrals of squared derivatives. For example, for a non-negative
integer $s$ and a function $f : \R \to \R$ with an $s^{th}$ derivative
$f^{(s)}$, the $s$-order Sobolev norm $\|\cdot\|_{H^s}$ is given by
$\|f\|_{H^s} = \int_\R \left( f^{(s)}(x) \right)^2 \, dx$
(when this quantity is finite). See Section \ref{sec:notation} for more general
definitions, and see \cite{leoni09Sobolev} for an introduction to Sobolev spaces.

Estimation of general Sobolev norms has a long history in nonparametric
statistics (e.g.,
\citet{schweder75windowVariance,ibragimov78nonparametricFunctionals,hall87SobolevNorms,bickel88squaredDensityDerivatives})
This line of work was motivated by the role of Sobolev norms in many semi- and
non-parametric problems, including density estimation, density functional
estimation, and regression, (see \citet{Tsybakov:2008:INE:1522486}, Section
1.7.1) where they dictate the convergence rates of estimators.
Despite this, to our
knowledge, these quantities have \emph{never} been studied in real data,
leaving an important gap between the theory and practice of nonparametric
statistics. We suggest this is in part due a lack of \emph{practical} estimators
for these quantities. For example, the only one of the above estimators that is
statistically minimax-optimal \citep{bickel88squaredDensityDerivatives} is
extremely difficult to compute in practice, requiring numerical integration over
each of $O(n^2)$ different kernel density estimates, where $n$ denotes the
sample size. We know of no estimators previously proposed for Sobolev inner
products and distances.

The {\bf main goal of this paper} is to propose and analyze a family of
computationally and statistically efficient estimators for Sobolev inner
products, norms, and distances. Our specific contributions are:
\begin{enumerate}
\item
We propose a family of nonparametric estimators for Sobolev norms, inner
products, and distances (Section \ref{sec:estimator}).
\item
We analyze the bias and variance of the estimators. Assuming the underlying
density functions have bounded support in $\R^D$ and lie in a Sobolev class of
sufficient smoothness parametrized by $s'$, we show that the estimator for
Sobolev quantities of order $s < s'$ converges to the true value at the
``parametric'' rate of $O(n\inv)$ in mean squared error when
$s' \geq 2s + D/4$, and at a slower rate of
$O \left( n^{\frac{8(s - s')}{4s' + D}} \right)$ otherwise.
(Section \ref{sec:finite_sample_bounds}).
\item
We derive asymptotic distributions for our estimators, and we use these to
derive tests for the general statistical problem of two-sample testing. We also
draw theoretical connections between our test and the recent work on
nonparametric two-sample testing. (Section \ref{subsec:two_sample_testing}).
\item
We validate our theoretical results on simulated data. (Section \ref{sec:empirical}).
\end{enumerate}

In terms of mean squared error, minimax lower bounds matching our convergence
rates over Sobolev or H\"older smoothness classes have been shown by
\citet{krishnamurthy2014renyiAndFriends} for $s = 0$ (i.e., $L^2$ quantities),
and \citet{birge95functionalLowerBounds} for Sobolev norms with
integer $s$. We conjecture but do not prove that our estimator is minimax
rate-optimal for all Sobolev quantities and $s \in [0, \infty)$.

As described in Section \ref{sec:practical}, our estimators are computable in
$O(n^{1 + \e})$ time using only basic matrix operations, where $n$ is the
sample size and $\e \in (0, 1)$ is a tunable parameter trading
statistical and computational efficiency; the smallest value of $\e$ at which
the estimator continues to be minimax rate-optimal approaches $0$ as we assume
more smoothness of the true density.

%
\section{Problem setup and notation}
\label{sec:notation}

Let $\X = [-\pi, \pi]^D$ and let $\mu$ denote the Lebesgue measure on $\X$. For
$D$-tuples $z \in \Z^D$ of integers, let $\psi_z \in L^2 = L^2(\X)$
\footnote{We suppress dependence on $\X$; all function spaces are over $\X$
except as discussed in Section \ref{sec:unbounded}.}
defined by
$\psi_z(x) = e^{-i \langle z, x \rangle}$
for all $x \in \X$ denote the $z^{th}$ element of the
$L^2$-orthonormal Fourier basis, and, for $f \in L^2$, let
$\tilde f(z)
  := \langle \psi_z, f \rangle_{L^2}
  = \int_\X \psi_z(x) \overline{f(x)} \, d\mu(x)$
denote the $z^{th}$ Fourier coefficient of $f$.
\footnote{Here, $\langle \cdot, \cdot \rangle$ denotes the dot product on
$\R^D$. For a complex number $c = a + bi$, $\overline{c} = a - bi$ denotes
the complex conjugate of $c$, and
$|c| = \sqrt{c\overline{c}} = \sqrt{a^2 + b^2}$ denotes the modulus of $c$.}
For any $s \in [0, \infty)$, define the Sobolev space
$H^s = H^s(\X) \subseteq L^2$ of order $s$ on $\X$ by
\footnote{When $D > 1$, $z^{2s} = \prod_{j = 1}^D z_j^{2s}$. For $z < 0$,
$z^{2s}$ should be read as $(z^2)^s$, so that $z^{2s} \in \R$ even when
$2s \notin \Z$. In the $L^2$ case, we use the convention that $0^0 = 1$.}
\begin{equation}
H^s
  = \left\{ f \in L^2 :
            \sum_{z \in \Z^D} z^{2s} \left| \tilde f(z) \right|^2 < \infty
    \right\}.
\label{eq:def_sobolev_space}
\end{equation}
Fix a known $s \in [0, \infty)$ and a unknown probability density functions
$p, q \in H^s$, and suppose we have $n$ IID samples $X_1,...,X_n \sim p$ and
$Y_1,\dots,Y_n \sim q$ from each of $p$ and $q$. We are interested in
estimating the inner product
\begin{equation}
\langle p, q \rangle_{H^s}
  := \sum_{z \in \Z^D} z^{2s} \tilde p(z) \overline{\tilde q(z)}
  \quad \mbox{ defined for all } \quad p, q \in H^s.
\label{eq:def_sobolev_inner_product}
\end{equation}
Estimating the inner product gives an estimate for the (squared) induced norm
and distance, since
\footnote{$\|p\|_{H^s}$ is \emph{pseudonorm} on $H^s$ because it fails to
distinguish functions identical almost everywhere up to additive constants; a
combination of $\|p\|_{L^2}$ and $\|p\|_{H^s}$ is used when a proper norm is
needed. However, since probability densities integrate to $1$,
$\|\cdot - \cdot\|_{H^s}$ \emph{is} a proper metric on the subset of
(almost-everywhere equivalence classes of) probability density functions in
$H^s$, which is important for two-sample testing  (see Section
\ref{subsec:two_sample_testing}). For simplicity, we use the terms ``norm'',
``inner product'', and ``distance'' for the remainder of the paper.}
\begin{equation}
\|p\|_{H^s}^2
 := \sum_{z \in \Z^D} z^{2s} \left| \tilde p(z) \right|^2
  = \langle p, p \rangle_{H^s}
  \quad \mbox{ and } \quad
  \|p - q\|_{H^s}^2
  = \|p\|_{H^s}^2 - 2 \langle p, q \rangle_{H^s} + \|q\|_{H^s}^2.
\label{eqs:sobolev_quantities}
\end{equation}
Since our theoretical results assume the samples from $p$ and $q$ are
independent, when estimating $\|p\|_{H^s}^2$, we split
the sample from $p$ in half to compute two independent estimates of $\tilde p$,
although this may not be optimal in practice.

For a more classical intuition, we note that, in the case $D = 1$ and
$s \in \{0,1,2,\dots\}$, (via Parseval's identity and the identity
$\tilde{f^{(s)}}(z) = (iz)^s \tilde f(z)$), that one can show the following:
$H^s$ includes the subspace of $L^2$ functions with at least $s$ derivatives in $L^2$ and,
if $f^{(s)}$ denotes the $s^{th}$ derivative of $f$
\begin{equation}
\|f\|_{H^s}^2
  = 2\pi \int_\X \left( f^{(s)}(x) \right)^2 \, dx
  = 2\pi \left\| f^{(s)} \right\|_{L^2}^2,
  \quad \forall f \in H^s.
\label{eq:sobolev_derivative_norm}
\end{equation}
In particular, when $s = 0$, $H^s = L^2$, $\|\cdot\|_{H^s} = \|\cdot\|_{L^2}$,
and $\langle \cdot, \cdot \rangle_{H^s} = \langle \cdot, \cdot \rangle_{L^2}$.
As we describe in the supplement, equation \eqref{eq:sobolev_derivative_norm}
and our results generalizes trivially to weak derivatives, as well as to
non-integer $s \in [0, \infty)$ via a notion of fractional derivative.

\subsection{Unbounded domains}
\label{sec:unbounded}
A notable restriction above is that $p$ and $q$ are supported in
$\X := [-\pi, \pi]^D$. In fact, our estimators and tests are well-defined and
valid for densities supported on arbitrary subsets of $\R^D$. In this case,
they act on the $2\pi$-periodic summation
$p_{2\pi} : [-\pi, \pi]^D \to [0, \infty]$ defined for $x \in \X$ by
$p_{2\pi}(x) := \sum_{z \in \Z^D} p(x + 2 \pi z)$, which is itself a
probability density function on $\X$. For example, the estimator for
$\|p\|_{H^s}$ will instead estimate $\|p_{2\pi}\|_{H^s}$, and the two-sample
test for distributions $p$ and $q$ will attempt to distinguish $p_{2\pi}$ from
$q_{2\pi}$. In most cases, this is not problematic; for example, for most
realistic probability densities, $p$ and $p_{2\pi}$ have similar orders of
smoothness, and $p_{2\pi} = q_{2\pi}$ if and only if $p = q$. However, there
are (meagre) sets of exceptions; for example, if $q$ is a translation of $p$ by
exactly $2\pi$, then $p_{2\pi} = q_{2\pi}$, and one can craft a highly
discontinuous function $p$ such that $p_{2\pi}$ is uniform on $\X$.
\citep{zygmund02trigSeries} These exceptions make it difficult to extend
theoretical results to densities with arbitrary support, but in practice, they
are fixed simply by randomly rescaling the data (similar to the approach of
\citet{chwialkowski15fastTwoSample}). If the densities have (known) bounded
support, they can simply be shifted and scaled to be supported on $\X$.

\section{Related work}
\label{sec:related_work}
There is a large body of work on estimating nonlinear functionals of probability
densities, with various generalizations in terms of the class of functionals
considered. Table \ref{tab:related_functionals} gives a subset of such work, for
functionals related to Sobolev quantities. As shown in Section
\ref{sec:notation}, the functional form we consider is a strict generalization
of $L^2$ norms, Sobolev norms, and $L^2$ inner products. It overlaps with, but
is neither a special case nor a generalization of the remaining functional forms
in the table.

\begin{table}
\centering
{\renewcommand{\arraystretch}{1.5} 
\begin{tabular}{|c|c|p{42mm}|}
\hline
Functional Name & Functional Form & References \\
\hline
$L^2$ norms & $\|p\|_{L^2}^2 = \int \left( p(x) \right)^2 \, dx$ &
\citet{schweder75windowVariance,gine08L2Norm} \\
\hline
(Integer) Sobolev norms & $\|p\|_{H^k}^2 = \int \left( p^{(k)}(x) \right)^2 \, dx$ & \citet{bickel88squaredDensityDerivatives} \\
\hline
Density functionals & $\int \phi(x, p(x)) \, dx$ & \citet{laurent92integralFunctionals, laurent96integralFunctionals} \\
\hline
Derivative functionals & $\int \phi(x, p(x), p'(x), \dots, p^{(k)}(x)) \, dx$ & \citet{birge95functionalLowerBounds} \\
\hline
$L^2$ inner products & $\langle p_1, p_2 \rangle_{L^2} = \int p_1(x) p_2(x) \, dx$ & \citet{krishnamurthy2014renyiAndFriends, krishnamurthy2014L2Divergence} \\
\hline
Multivariate functionals & $\int \phi(x, p_1(x), \dots,p_k(x)) \, dx$ & \citet{singh14densityFunctionals, kandasamy15vonMises} \\
\hline
\end{tabular}
}
\caption{Some related functional forms for which estimators for which
nonparametric estimators have been developed and analyzed. $p,p_1,...,p_k$ are
unknown probability densities, from each of which we draw $n$ IID samples,
$\phi$ is a known real-valued measurable function, and $k$ is a non-negative
integer.}
\label{tab:related_functionals}
\end{table}
%

Nearly all of the above approaches compute an optimally smoothed kernel density
estimate and then perform bias corrections based on Taylor series expansions of
the functional of interest. They typically consider distributions with
densities that are $\beta$-H\"older continuous and satisfy periodicity
assumptions of order $\beta$ on the boundary of their support, for some
constant $\beta > 0$ (see, for example, Section 4 of
\citet{krishnamurthy2014renyiAndFriends} for details of these assumptions). The
Sobolev class we consider is a strict superset of this H\"older class,
permitting, for example, certain ``small'' discontinuities.
In this regard, our results are slightly more general than most of these prior
works.

Finally, there is much recent work on estimating entropies, divergences, and
mutual informations, using methods based on kernel density estimates
\citep{singh14RenyiDivergence,singh14densityFunctionals,moon16improvingConvergence,krishnamurthy2014renyiAndFriends,krishnamurthy2014L2Divergence,kandasamy15vonMises}
or $k$-nearest neighbor statistics
\citep{leonenko08RenyiEntropy,poczos11alphaDivergence,moon14divergencesEnsemble,moon14divergencesConfidence}.
In contrast, our estimators are more similar to orthogonal series density
estimators, which are computationally attractive because they require no
pairwise operations between samples. However, they require quite different
theoretical analysis; unlike prior work, our estimator is constructed and
analyzed entirely in the frequency domain, and then related to the data domain
via Parseval's identity. We hope our analysis can be adapted to analyze new,
computationally efficient information theoretic estimators.

\section{Motivation and construction of our estimator}
\label{sec:estimator}
For a non-negative integer parameter $Z_n$ (to be specified later), let
\begin{equation}
p_n
  := \sum_{\|z\|_\infty \leq Z_n} \tilde p(z) \psi_z
\quad \mbox{ and } \quad
q_n
  := \sum_{\|z\|_\infty \leq Z_n} \tilde q(z) \psi_z
  \quad \mbox{ where } \quad
\|z\|_\infty := \max_{j \in \{1,\dots,D\}} z_j
\label{eq:def_trig_approx}
\end{equation}
denote the $L^2$ projections of $p$ and $q$, respectively, onto the linear
subspace spanned by the $L^2$-orthonormal family
$\F_n := \{\psi_z : z \in \Z^D, |z| \leq Z_n\}$. Note that, since
$\tilde{\psi_z}(y) = 0$ whenever $y \neq z$, the Fourier basis has the special
property that it is orthogonal in $\langle \cdot, \cdot \rangle_{H^s}$ as well.
Hence, since $p_n$ and $q_n$ lie in the span of $\F_n$ while $p - p_n$ and
$q - q_n$ lie in the span of $\{\psi_z : z \in \Z\} \sminus \F_n$,
$\langle p - p_n, q_n \rangle_{H^s} = \langle p_n, q - q_n \rangle_{H^s} = 0$.
Therefore,
\begin{align}
\langle p, q \rangle_{H^s}
\notag
& = \langle p_n, q_n \rangle_{H^s}
  + \langle p - p_n, q_n \rangle_{H^s}
  + \langle p_n, q - q_n \rangle_{H^s}
  + \langle p - p_n, q - q_n \rangle_{H^s} \\
\label{eq:orthogonal_decomposition}
& = \langle p_n, q_n \rangle_{H^s}
  + \langle p - p_n, q - q_n \rangle_{H^s}.
\end{align}
We propose an unbiased estimate of
$S_n
  := \langle p_n, q_n \rangle_{H^s}
  = \sum_{\|z\|_\infty \leq Z_n} z^{2s} \tilde p_n(z) \overline{\tilde q_n(z)}$.
Notice that Fourier coefficients of $p$ are the expectations
$\tilde p(z) = \E_{X \sim p} \left[ \psi_z(X) \right]$. Thus,
$\hat p(z) := \frac{1}{n} \sum_{j = 1}^n \psi_z(X_j)$ and
$\hat q(z) := \frac{1}{n} \sum_{j = 1}^n \psi_z(Y_j)$ are independent unbiased
estimates of $\tilde p$ and $\tilde q$, respectively. Since $S_n$ is bilinear
in $\tilde p$ and $\tilde q$, the plug-in estimator for $S_n$ is unbiased. That
is, {\bf our estimator} for $\langle p, q \rangle_{H^s}$ is
\begin{equation}
\hat S_n
  := \sum_{\|z\|_\infty \leq Z_n} z^{2s} \hat p(z) \overline{\hat q(z)}.
\label{eq:sobolev_final_estimator}
\end{equation}

\section{Finite sample bounds}
\label{sec:finite_sample_bounds}
Here, we present our main theoretical results, bounding the bias, variance, and
mean squared error of our estimator for finite $n$.

By construction, our estimator satisfies
\[\E \left[ \hat S_n \right]
  = \sum_{\|z\|_\infty \leq Z_n} z^{2s} \E \left[ \hat p(z) \right]
                                 \overline{\E \left[ \hat q(z) \right]}
  = \sum_{\|z\|_\infty \leq Z_n} z^{2s} \tilde p_n(z) \overline{\tilde q_n(z)}
  = S_n.\]
Thus, via \eqref{eq:orthogonal_decomposition} and Cauchy-Schwarz, the
bias of the estimator $\hat S_n$ satisfies
\begin{equation}
\left| \E \left[ \hat S_n \right] - \langle p, q \rangle_{H^s} \right|
  = \left| \langle p - p_n, q - q_n \rangle_{H^s} \right|
  \leq \sqrt{\left\| p - p_n \right\|_{H^s}^2
             \left\| q - q_n \right\|_{H^s}^2}.
\label{eq:_bias_intermediate}
\end{equation}
$\left\| p - p_n \right\|_{H^s}$ is the error of approximating $p$ by an
order-$Z_n$ trigonometric polynomial, a classic problem in approximation
theory, for which Theorem 2.2 of \citet{kreiss79Fourierstability} shows:
\begin{equation}
\mbox{ if } p \in H^{s'} \mbox{ for some } s' > s,
  \quad \mbox{ then } \quad
  \left\| p - p_n \right\|_{H^s}
  \leq \|p\|_{H^{s'}}Z_n^{s - s'}.
\label{ineq:trig_approximation_error}
\end{equation}
In combination with \eqref{eq:_bias_intermediate}, this implies the following
bound on the bias of our estimator:
\begin{theorem}
{\bf (Bias bound)}
If $p, q \in H^{s'}$ for some $s' > s$, then, for
$C_B := \|p\|_{H^{s'}} \|q\|_{H^{s'}}$,
\begin{equation}
\left| \E \left[ \hat S_n \right] - \langle p, q \rangle_{H^s} \right|
  \leq C_B Z_n^{2(s - s')}
\label{eq:overall_bias}
\end{equation}
\label{thm:bias_bound}
\end{theorem}
Hence, the bias of $\hat S_n$ decays polynomially in $Z_n$, with a power
depending on the ``extra'' $s' - s$ orders of smoothness available. On the
other hand, as we increase $Z_n$, the number of frequencies at which we
estimate $\hat p$ increases, suggesting that the variance of the estimator will
increase with $Z_n$. Indeed, this is expressed in the following bound on the
variance of the estimator.

\begin{theorem}
{\bf (Variance bound)}
If $p, q \in H^{s'}$ for some $s' \geq s$, then
\begin{equation}
\Var \left[ \hat S_n \right]
  \leq 2C_1 \frac{Z_n^{4s + D}}{n^2} + \frac{C_2}{n},
\label{ineq:variance_bound}
\end{equation}
where $C_1$ and $C_2$ are the constants (in $n$)
\begin{equation}
C_1
  := \frac{2^D\Gamma(4s + 1)}{\Gamma(4s + D + 1)} \|p\|_{L^2} \|q\|_{L^2}
\label{constant:C_1}
\end{equation}
and
$C_2
  := \left( \|p\|_{H^s} + \|q\|_{H^s} \right) \|p\|_{W^{2s,4}} \|q\|_{W^{2s,4}}
      + \|p\|_{H^s}^4 \|q\|_{H^s}^4$.
\label{thm:variance_bound}
\end{theorem}
%
The proof of Theorem \ref{thm:variance_bound} is perhaps the most significant
theoretical contribution of this work. Due to space constraints, the proof is
given in the appendix. Combining Theorems \ref{thm:bias_bound} and
\ref{thm:variance_bound} gives a bound on the mean squared error (MSE) of
$\hat S_n$ via the usual decomposition into squared bias and variance:
\begin{corollary}
{\bf (Mean squared error bound)}
If $p, q \in H^{s'}$ for some $s' > s$, then
\label{corr:MSE_bound}
\begin{equation}
\E \left[ \left( \hat S_n - \langle p, q \rangle_{H^s} \right)^2 \right]
  \leq C_B^2 Z_n^{4(s - s')}
  + 2C_1 \frac{Z_n^{4s + D}}{n^2} + \frac{C_2}{n}.
\label{ineq:MSE_bound_general_Z_n}
\end{equation}
If, furthermore, we choose $Z_n \asymp n^{\frac{2}{4s' + D}}$ (optimizing the
rate in inequality \ref{ineq:MSE_bound_general_Z_n}), then
\begin{equation}
\E \left[ \left( \hat S_n - \langle p, q \rangle_{H^2} \right)^2 \right]
  \asymp n^{\max\left\{ \frac{8(s - s')}{4s' + D}, -1 \right\}}.
\label{ineq:asymptotic_MSE}
\end{equation}
\end{corollary}
Corollary \ref{corr:MSE_bound} recovers the phenomenon discovered by
\citet{bickel88squaredDensityDerivatives}: when $s' \geq 2s + \frac{D}{4}$,
the minimax optimal MSE decays at the ``semi-parametric'' $n\inv$ rate,
whereas, when $s' \in \left( s, 2s + \frac{D}{4} \right)$, the MSE decays at a
slower rate. Also, the estimator is $L^2$-consistent if $Z_n \to \infty$ and
$Z_n n^{-\frac{2}{4s + D}} \to 0$ as $n \to \infty$. This is useful in
practice, since $s$ is known but $s'$ is not.

\section{Asymptotic distributions}
\label{sec:asymptotic_results}
In this section, we derive the asymptotic distributions of our estimator in two
cases: (1) the inner product estimator and (2) the distance estimator in the
case $p = q$. These results provide confidence intervals and two-sample tests
without computationally intensive resampling. While (1) is more general in that
it can be used with (\ref{eqs:sobolev_quantities}) to bound the asymptotic
distributions of the norm and distance estimators, (2) provides a more precise
result leading to a more computationally and statistically efficient two-sample
test. Proofs are given in the supplementary material.

Theorem \ref{thm:asymp_normal} shows that our estimator has a normal asymptotic
distribution, assuming $Z_n \to \infty$ slowly enough as $n \to \infty$, and
also gives a consistent estimator for its asymptotic variance. From this, one
can easily estimate asymptotic confidence intervals for inner products, and
hence also for norms.

\begin{theorem}
{\bf (Asymptotic normality)}
Suppose that, for some $s' > 2s + \frac{D}{4}$, $p, q \in H^{s'}$, and
suppose $Z_n n^{\frac{1}{4(s - s')}} \to \infty$ and
$Z_n n^{-\frac{1}{4s + D}} \to 0$ as $n \to \infty$. Then, $\hat S_n$ is
asymptotically normal with mean $\langle p, q \rangle_{H^s}$. In particular, for
$j \in \{1, \dots, n\}$ and $z \in \Z^D$ with $\|z\|_\infty \leq Z_n$, define
$W_{j,z} := z^s e^{izX_j}$ and $V_{j,z} := z^s e^{izY_j}$, so that $W_j$ and
$V_j$ are column vectors in $\R^{(2Z_n)^D}$. Let
$\overline W := \frac{1}{n} \sum_{j = 1}^n W_j,
  \overline V := \frac{1}{n} \sum_{j = 1}^n V_j
  \in \R^{(2Z_n)^D},$
\[\Sigma_W := \frac{1}{n} \sum_{j = 1}^n (W_j - \overline W) (W_j - \overline W)^T,
  \quad \mbox{ and } \quad
  \Sigma_V := \frac{1}{n} \sum_{j = 1}^n (V_j - \overline V) (V_j - \overline V)^T
  \in \R^{(2Z_n)^D \times (2Z_n)^D}\]
denote the empirical means and covariances of $W$ and $V$, respectively. Then,
for
\[\hat \sigma_n^2
  := \begin{bmatrix}
      \overline V \\
      \overline W
    \end{bmatrix}^T
    \begin{bmatrix}
      \Sigma_W & 0 \\
      0 & \Sigma_V
    \end{bmatrix}
    \begin{bmatrix}
      \overline V \\
      \overline W
    \end{bmatrix},
  \quad \mbox{ we have } \quad
\sqrt{n} \left( \frac{\hat S_n - \langle p, q \rangle_{H^s}}{\hat\sigma_n} \right)
  \stackrel{D}{\to} \mathcal{N}(0, 1),\]
where $\stackrel{D}{\to}$ denotes convergence in distribution.
\label{thm:asymp_normal}
\end{theorem}
Since distances can be written as a sum of three inner products (Eq.
(\ref{eqs:sobolev_quantities})), Theorem \ref{thm:asymp_normal} might suggest
an asymptotic normal distribution for Sobolev distances. However, extending
asymptotic normality from inner products to their sum requires that the three
estimates be independent, and hence that we split data between the three
estimates. This is inefficient in practice and somewhat unnatural, as we know,
for example, that distances should be non-negative. For the particular case
$p = q$ (as in the null hypothesis of two-sample testing), the following
theorem
\footnote{This result is closely related to Proposition 4 of
\citet{chwialkowski15fastTwoSample}. However, in their situation, $s = 0$ and
the set of test frequencies is fixed as $n \to \infty$, whereas our set is
increasing.}
provides a more precise asymptotic ($\chi^2$) distribution of our Sobolev
distance estimator, after an extra decorrelation step. This gives, for example,
a more powerful two-sample test statistic (see Section
\ref{subsec:two_sample_testing} for details).
\begin{theorem}
{\bf (Asymptotic null distribution)}
Suppose that, for some $s' > 2s + \frac{D}{4}$, $p, q \in H^{s'}$, and
suppose $Z_n n^{\frac{1}{4(s - s')}} \to \infty$ and
$Z_n n^{-\frac{1}{4s + D}} \to 0$ as $n \to \infty$.
For $j \in \{1,\dots,n\}$ and $z \in \Z^D$ with $\|z\|_{\infty} \leq Z_n$,
define $W_{j, z} := z^s \left( e^{-izX_j} - e^{-izY_j} \right)$, so that $W_j$
is a column vector in $\R^{(2Z_n)^D}$. Let
\[\overline W := \frac{1}{n} \sum_{j = 1}^n W_j \in \R^{(2Z_n)^D}
  \quad \mbox{ and } \quad
  \Sigma
  := \frac{1}{n} \sum_{j = 1}^n \left( W_j - \overline W \right)
                                \left( W_j - \overline W \right)^T
  \in \R^{(2Z_n)^D \times (2Z_n)^D}\]
denote the empirical mean and covariance of $W$, and define
$T := n \overline W^T \Sigma\inv \overline W$. Then, if $p = q$, then
\[Q_{\chi^2((2Z_n)^D)}(T) \stackrel{D}{\to} \operatorname{Uniform}([0, 1])
  \quad \mbox{ as } \quad
  n \to \infty,\]
where $Q_{\chi^2(d)} : [0, \infty) \to [0, 1]$ denotes the quantile function
(inverse CDF) of the $\chi^2$ distribution $\chi^2(d)$ with $d$ degrees of
freedom.
\label{thm:asymptotic_test_distribution}
\end{theorem}
Let $\hat M$ denote our estimator for $\|p - q\|_{H^s}$ (i.e., plugging
$\hat S_n$ into \eqref{eqs:sobolev_quantities}).
While Theorem \ref{thm:asymptotic_test_distribution} immediately provides a
\emph{valid} two-sample test of desired level, it is not immediately clear
how this relates to $\hat M$, nor is there any suggestion of why the test
statistic ought to be a good (i.e., consistent) one. Some intuition is as
follows. Notice that $\hat M = \overline W^T \overline W$. Since, by the central
limit theorem, $\overline W$ has a normal asymptotic distribution, if the
components of $\overline W$ were uncorrelated (and $Z_n$ were fixed), we would
expect $n \hat M$ to have an asymptotic $\chi^2$ distribution with
$(2Z_n)^D$ degrees of freedom. However, because we use the same data to compute
each component of $\hat M$, they are \emph{not} typically uncorrelated, and so
the asymptotic distribution of $\hat M$ is difficult to derive. This motivates
the statistic
$T = \left( \sqrt{\Sigma_W\inv} \overline W \right)^T
            \sqrt{\Sigma_W\inv} \overline W$,
since the components of $\sqrt{\Sigma_W\inv} \overline W$ \emph{are}
(asymptotically) uncorrelated.

\section{Parameter selection and statistical/computational trade-off}
\label{sec:practical}
Here, we give statistical and computational considerations for choosing the
smoothing parameter $Z_n$.

{\bf Statistical perspective:} In practice, of course, we do not typically know
$s'$, so we cannot simply set $Z_n \asymp n^{\frac{2}{4s' + D}}$, as suggested
by the mean squared error bound (\ref{ineq:asymptotic_MSE}). Fortunately (at
least for ease of parameter selection), when $s' \geq 2s + \frac{D}{4}$, the
dominant term of (\ref{ineq:asymptotic_MSE}) is $C_2/n$ for
$Z_n \asymp n^{-\frac{1}{4s + D}}$. Hence if we are willing to assume that the
density has at least $2s + \frac{D}{4}$ orders of smoothness (which may be a
mild assumption in practice), then we achieve \emph{statistical optimality} (in
rate) by setting $Z_n \asymp n^{-\frac{1}{4s + D}}$, which depends only on
known parameters. On the other hand, the estimator can continue to benefit from
additional smoothness \emph{computationally}.

{\bf Computational perspective}
One attractive property of the estimator discussed is its computational
simplicity and efficiency with respect to $n$, in low dimensions. Most
competing nonparametric estimators, such as kernel-based or nearest-neighbor
methods, either take $O(n^2)$ time or rely on complex data structures such as
$k$-d trees or cover trees \citep{ram09linearTime} for $O(2^D n\log n)$ time
performance. Since computing the estimator takes $O(nZ_n^D)$ time and
$O(Z_n^D)$ memory (that is, the cost of estimating each of $(2Z_n)^D$ Fourier
coefficients by an average), a statistically optimal choice of $Z_n$ gives a
runtime of $O \left( n^{\frac{4s' + 2D}{4s' + D}} \right)$.
Since the estimate requires only a vector outer product, exponentiation, and
averaging, the constants involved are small and computations parallelize
trivially over frequencies and data.

Under severe computational constraints, for very large data sets, or if $D$ is
large relative to $s'$, we can reduce $Z_n$ to trade off statistical for
computational efficiency. For example, if we want an estimator
with runtime $O(n^{1 + \theta})$ and space requirement $O(n^\theta)$ for
some $\theta \in \left( 0, \frac{2D}{4s' + D} \right)$, setting
$Z_n \asymp n^{\theta/D}$ still gives a consistent estimator, with mean squared
error of the order
$O \left( n^{\max\{\frac{4\theta(s - s')}{D}, -1\}} \right)$.

Kernel- or nearest-neighbor-based methods, including nearly all of the methods
described in Section \ref{sec:related_work}, tend to require storing previously
observed data, resulting in $O(n)$ space requirements. In contrast, orthogonal
basis estimation requires storing only $O(Z_n^D)$ estimated Fourier
coefficients. The estimated coefficients can be incrementally updated with each
new data point, which may make the estimator or close approximations feasible
in streaming settings.

\section{Experimental results}
\label{sec:empirical}

In this section, we use synthesized data to demonstrate the effectiveness of our methods.
A MATLAB implementation of our estimators, two-sample tests, and experiments is
available at \url{https://github.com/sss1/SobolevEstimation}.
For all experiments, we use $10,100,1000,10000,100000$ samples for estimation.

We first test our estimators on 1D $L_2$ distances.
Figure~\ref{fig:distance_1d_gaussian_different_mean} shows estimated distance between $\mathcal{N}\left(0,1\right)$ and $\mathcal{N}\left(1,1\right)$;
Figure~\ref{fig:distance_1d_gaussian_different_var} shows estimated distance between $\mathcal{N}\left(0,1\right)$ and $\mathcal{N}\left(0,4\right)$;
Figure~\ref{fig:distance_unif_different_range} shows estimated distance between Unif $\left[0,1\right]$ and Unif$\left[0.5,1.5\right]$;
Figure~\ref{fig:distance_unif_tri} shows estimated distance between $\left[0,1\right]$ and a triangular distribution whose density is highest at $x=0.5$.
Error bars indicate asymptotic $95\%$ confidence intervals based on Theorem
\ref{thm:asymp_normal}. These experiments suggest $10^5$ samples is sufficient
to estimate $L_2$ distances with high confidence. Note that we need fewer
samples to estimate Sobolev quantities of Gaussians than, say, of uniform
distributions, consistent with our theory, since (infinitely differentiable)
Gaussians are smoothier than (discontinuous) uniform distributions.

Next, we test our estimators on $L_2$ distances of multivariate distributions.
Figure~\ref{fig:distance_3d_gaussian_different_mean} shows estimated distance between $\mathcal{N}\left(\left[0,0,0\right],\mathbf{I}\right)$ and $\mathcal{N}\left(\left[1,1,1\right],\mathbf{I}\right)$;
Figure~\ref{fig:distance_3d_gaussian_different_var} shows estimated distance between $\mathcal{N}\left(\left[0,0,0\right],\mathbf{I}\right)$ and $\mathcal{N}\left(\left[0,0,0\right],4\mathbf{I}\right)$.
Again, these experiments show that our estimators can also handle multivariate distributions.

Lastly, we test our estimators for $H^s$ norms.
Figure~\ref{fig:norm_gaussian} shows estimated $H^0$ norm of $\mathcal{N}\left(0,1\right)$ and 
Figure~\ref{fig:norm_gaussian_derivative} shows $H^1$ norm of $\mathcal{N}\left(0,1\right)$.
Notice that we need fewer samples to estimate $H_0$ than $H^1$, which verifies our theory.

\begin{figure*}[t]
	\centering
	\begin{subfigure}[t]{0.22\textwidth}
		\centering
		\includegraphics[width=\textwidth]{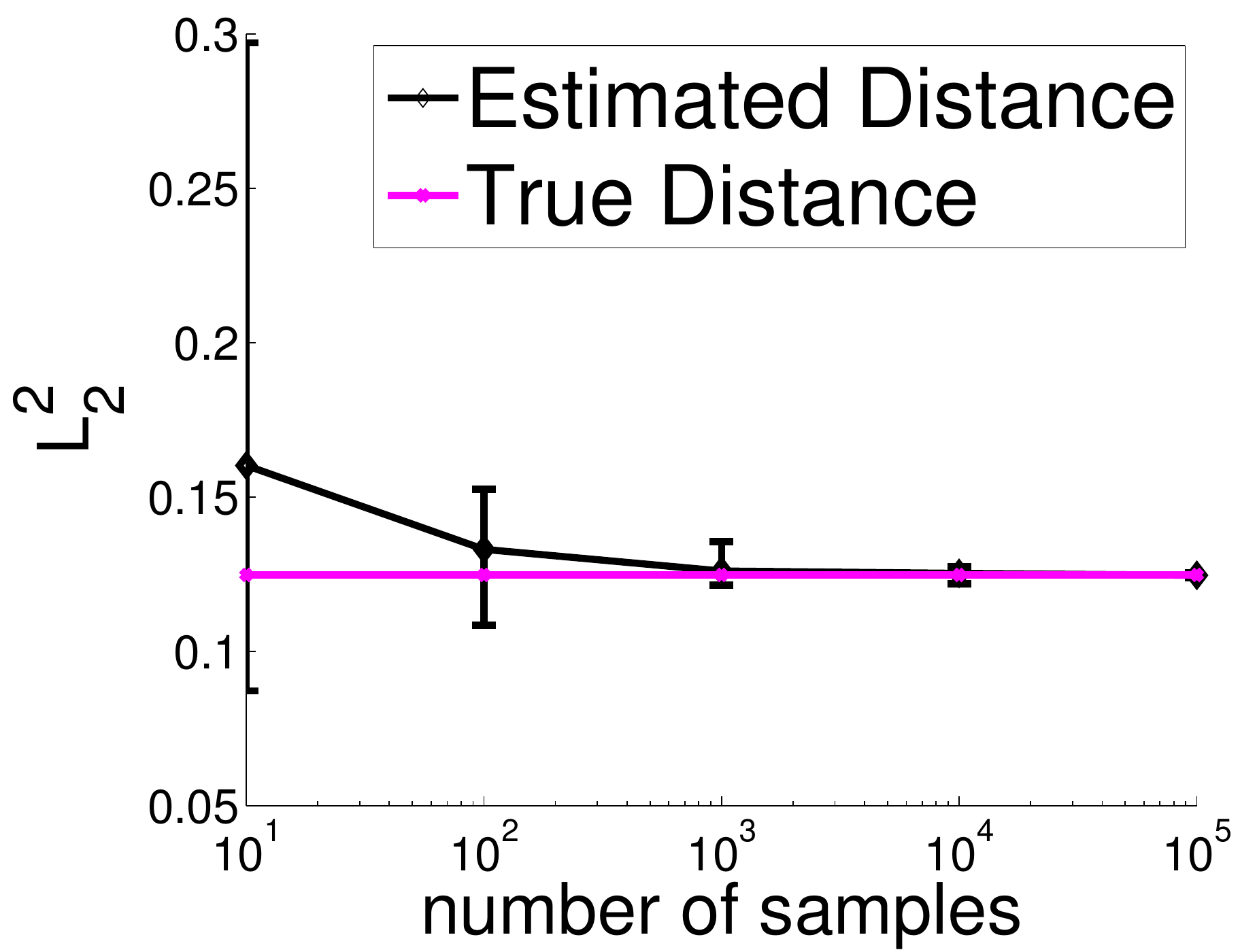}
		\caption{Two 1D Gaussian distributions  with different means.}
		\label{fig:distance_1d_gaussian_different_mean}
	\end{subfigure}
	\quad
	\begin{subfigure}[t]{0.22\textwidth}
		\centering
		\includegraphics[width=\textwidth]{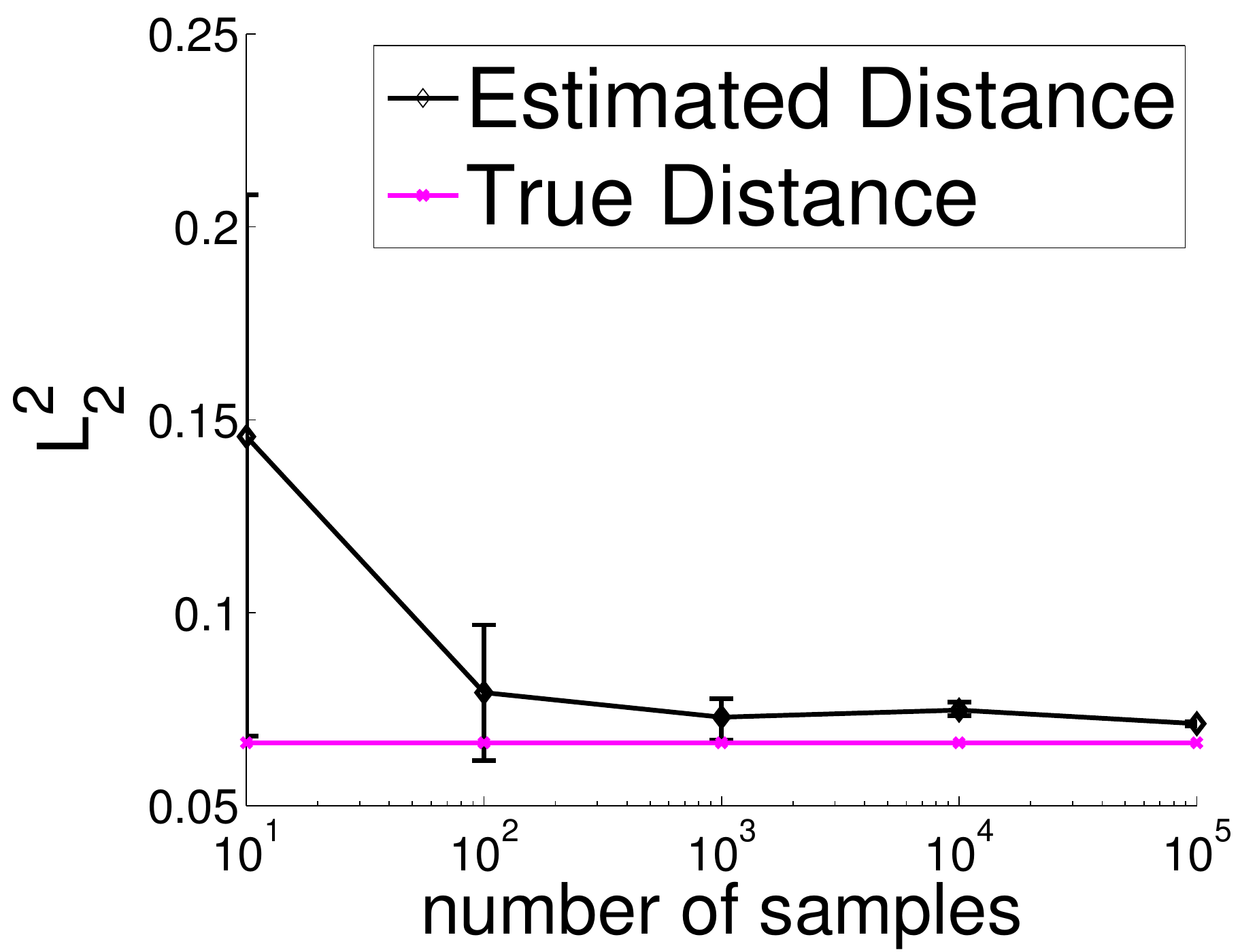}
		\caption{Two 1D Gaussian distributions with different variance.}
		\label{fig:distance_1d_gaussian_different_var}
	\end{subfigure}
	\quad
	\begin{subfigure}[t]{0.22\textwidth}
		\centering
		\includegraphics[width=\textwidth]{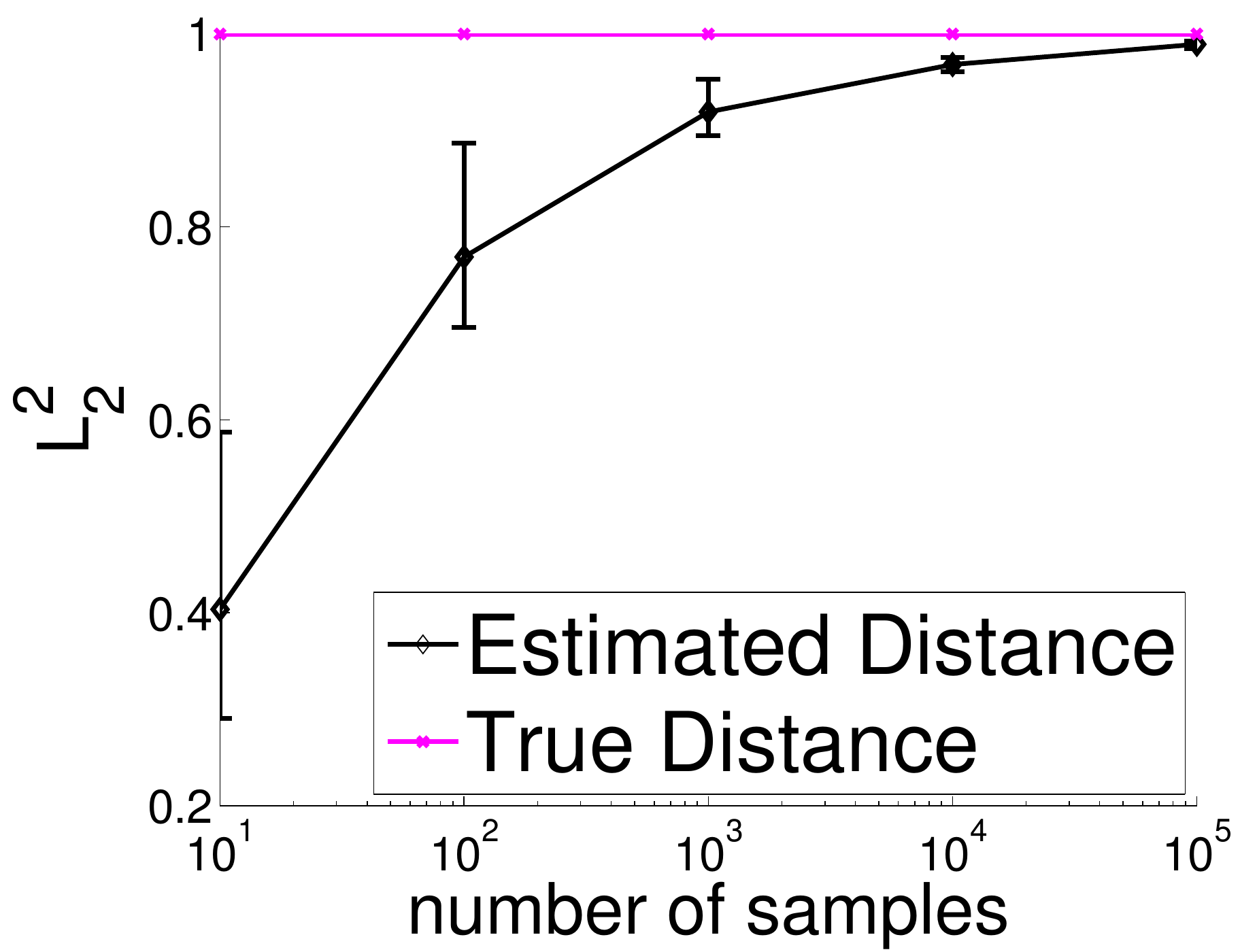}
		\caption{Two uniform distributions with different ranges.}
		\label{fig:distance_unif_different_range}
	\end{subfigure}
	\quad
	\begin{subfigure}[t]{0.22\textwidth}
		\centering
		\includegraphics[width=\textwidth]{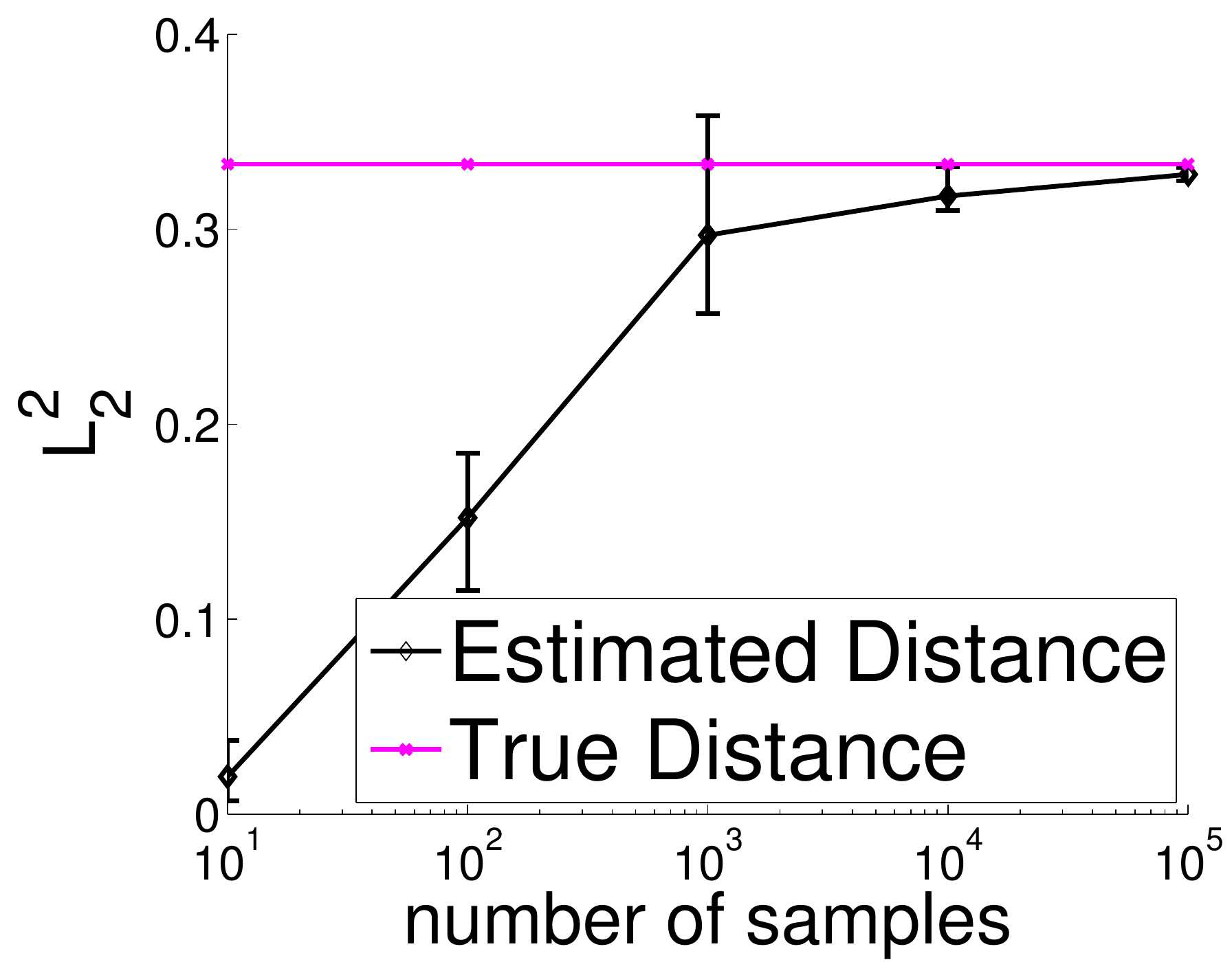}
		\caption{One uniform distribution and one triangular distribution.}
		\label{fig:distance_unif_tri}
	\end{subfigure}
\vspace{-0.3cm}
\end{figure*}

\begin{figure*}[t]
	\centering
	\begin{subfigure}[t]{0.22\textwidth}
		\centering
		\includegraphics[width=\textwidth]{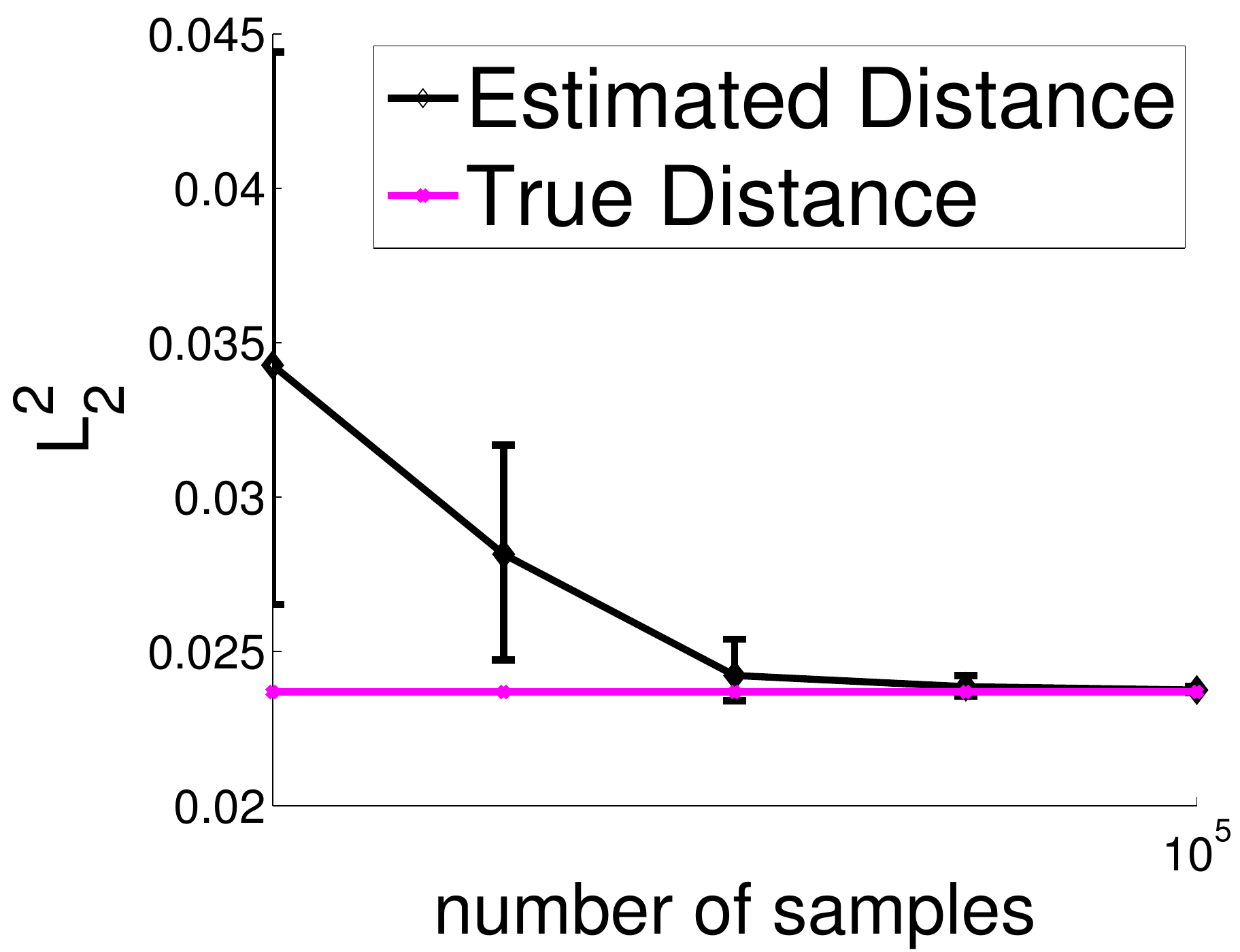}
		\caption{Two 3D Gaussian distributions with different means.}
		\label{fig:distance_3d_gaussian_different_mean}
	\end{subfigure}
	\quad
	\begin{subfigure}[t]{0.22\textwidth}
		\includegraphics[width=\textwidth]{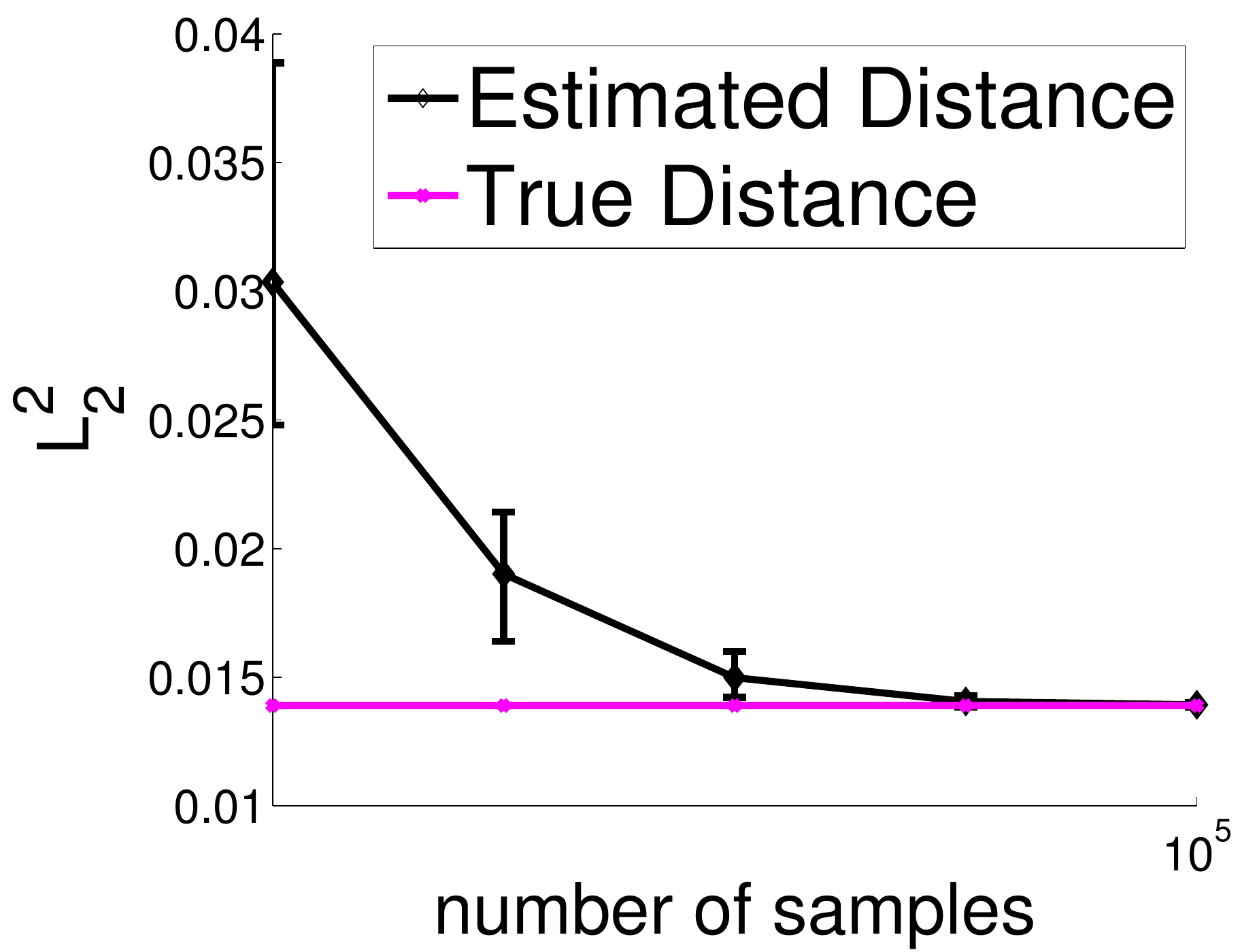}
		\caption{
		Two 3D Gaussian distributions with different variance.
		}
		\label{fig:distance_3d_gaussian_different_var}
	\end{subfigure}
	\quad
		\begin{subfigure}[t]{0.22\textwidth}
			\centering
			\includegraphics[width=\textwidth]{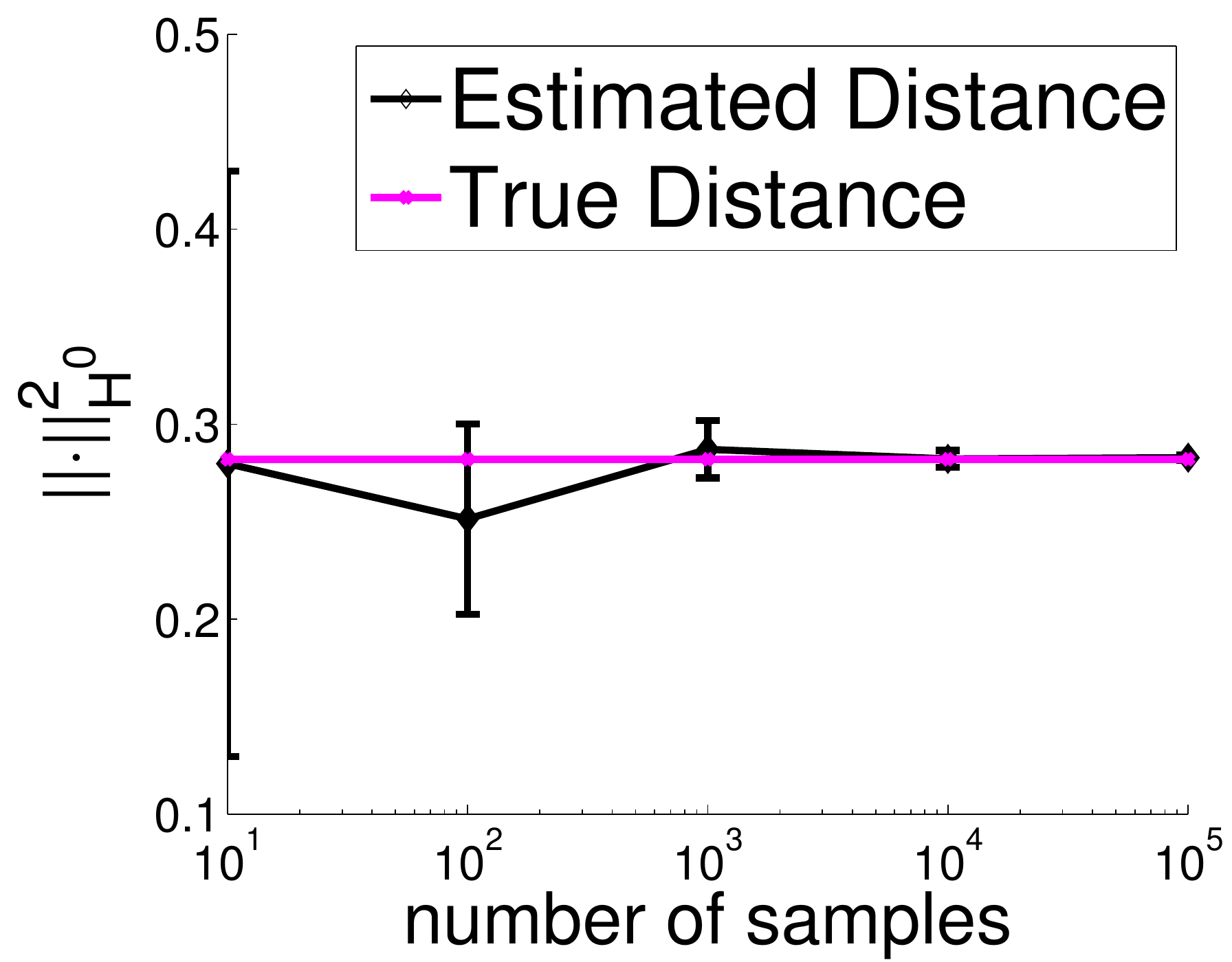}			
					\caption{
					Estimation of $H^0$ norm of $\mathcal{N}\left(0,1\right)$.
					}
					\label{fig:norm_gaussian}
		\end{subfigure}
		\quad
		\begin{subfigure}[t]{0.22\textwidth}
			\centering
			\includegraphics[width=\textwidth]{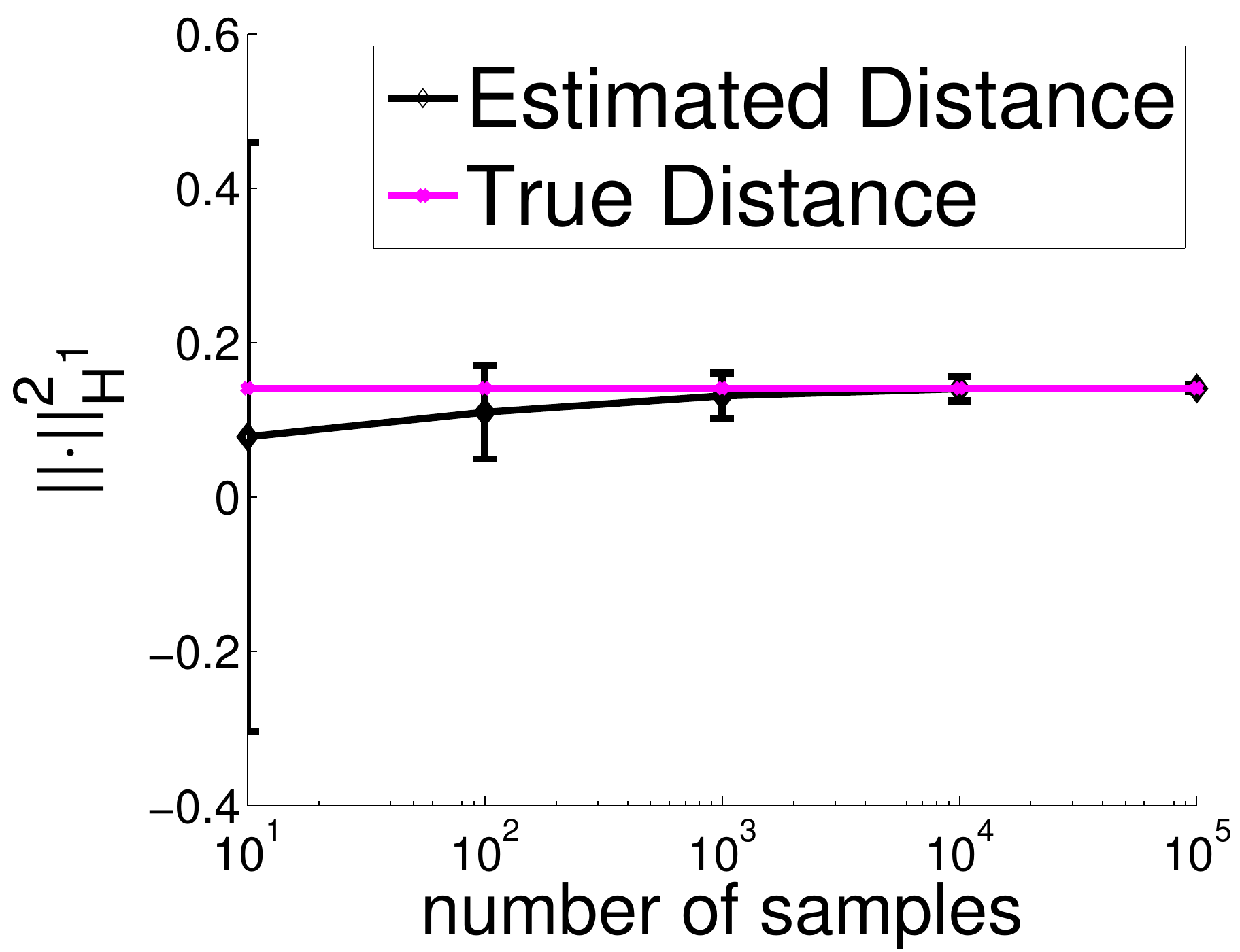}	
			\caption{
			Estimation of $H^1$ norm of $\mathcal{N}\left(0,1\right)$.
			}
			\label{fig:norm_gaussian_derivative}
		\end{subfigure}
\vspace{-0.3cm}
\end{figure*}

\section{Connections to two-sample testing}
\label{subsec:two_sample_testing}

Here, we discuss the use of our estimator in two-sample testing. There is a
large literature on nonparametric two-sample testing, but we discuss only some
recent approaches with theoretical connections to ours.

Let $\hat M$ denote our estimate of the Sobolev distance, consisting of
plugging $\hat S$ into equation (\ref{eqs:sobolev_quantities}). Since
$\|\cdot - \cdot\|_{H^s}$ is a metric on the space of probability density
functions in $H^s$, computing $\hat M$ leads naturally to a two-sample test on
this space. Theorem \ref{thm:asymptotic_test_distribution} suggests an
asymptotic test, which is computationally preferable to a permutation test. In
particular, for a desired Type I error rate $\alpha \in (0, 1)$ our test
rejects the null hypothesis $p = q$ if and only if
$Q_{\chi^2(2Z_n^D)}(T) < \alpha$.

When $s = 0$, this approach is closely related to several two-sample
tests in the literature based on comparing empirical characteristic
functions (CFs). Originally, these tests
\citep{heathcote1972test,epps86twoSampleECF} computed the same statistic $T$
with a \emph{fixed} number of random $\R^D$-valued frequencies instead of
deterministic $\Z^D$-valued frequencies. This test runs in linear time, but is
not generally consistent, since the two CFs need not differ almost everywhere.
Recently, \citet{chwialkowski15fastTwoSample} suggested using \emph{smoothed
CFs}, i.e., the convolution of the CF with a universal smoothing kernel $k$.
This is computationally easy (due to the convolution theorem) and, when
$p \neq q$, $(\tilde p * k)(z) \neq (\tilde q * k)(z)$ for almost all
$z \in \R^D$, reducing the need for carefully choosing test frequencies.
Furthermore, this test is almost-surely consistent under very general
alternatives. However, it is not clear what sort of assumptions would allow
finite sample analysis of the power of their test. Indeed, the convergence as
$n \to \infty$ can be arbitrarily slow, depending on the random test
frequencies used. Our analysis instead uses the assumption $p, q \in H^{s'}$ to
ensure that small, $\Z^D$-valued frequencies contain most of the power of
$\tilde p$.
\footnote{Note that smooth CFs can be used in our test by replacing $\hat p(z)$
with $\frac{1}{n} \sum_{j = 1}^n e^{-iz X_j} k(x)$, where $k$ is the inverse
Fourier transform of a characteristic kernel. However, smoothing seems less
desirable under Sobolev assumptions, as it spreads the power of the CF away
from small $\Z^D$-valued frequencies where our test focuses.}

These fixed-frequency approaches can be thought of as the extreme point
$\theta = 0$ of the computational/statistical trade-off described in section
\ref{sec:practical}: they are computable in linear time and (with
smoothing) are strongly consistent, but do not satisfy finite-sample bounds
under general conditions.

At the other extreme ($\theta = 1$) are MMD-based tests of
\citet{gretton06MMD_NIPS,gretton12MMD_JMLR}, which utilize the entire spectrum
$\tilde p$. These tests are statistically powerful and have strong guarantees
for densities in an RKHS, but have $O(n^2)$ computational complexity.
\footnote{Fast MMD approximations have been proposed, including the Block MMD,
\citep{zaremba13blockMMD} FastMMD, \citep{zhao15fastMMD} and $\sqrt{n}$
sub-sampled MMD, but these lack the statistical guarantees of MMD.} The
computational/statistical trade-off discussed in Section
\ref{sec:practical} can be thought of as an interpolation (controlled
by $\theta$) of these approaches, with runtime in the case $\theta = 1$
approaching quadratic for large $D$ and small $s'$.

\section{Conclusions and future work}
In this paper, we proposed nonparametric estimators for Sobolev inner products,
norms and distances of probability densities, for which we derived finite-sample
bounds and asymptotic distributions.

A natural follow-up question to our work is whether estimating smoothness of a
density can guide the choice of smoothing parameters in nonparametric
estimation. For some problems, such as estimating functionals of a density,
this may be especially useful, since no error metric is typically available for
cross-validation. Even when cross-validation is an option, as in density
estimation or regression, estimating smoothness may be faster, or may suggest
an appropriate range of parameter values.

\appendix

\section{Proof of Variance Bound}
\label{sec:variance_bound}
\stepcounter{theorem}
\begin{theorem}
{\bf (Variance Bound)}
If $p, q \in H^{s'}$ for some $s' > s$, then
\begin{equation}
\Var \left[ \hat S_n \right]
  \leq 2C_1 \frac{Z_n^{4s + D}}{n^2} + \frac{C_2}{n},
\label{ineq:variance_bound_appendix}
\end{equation}
where $C_1$ and $C_2$ are the constants (in $n$)
\[C_1 := \frac{2^D\Gamma(4s + 1)}{\Gamma(4s + D + 1)} \|p\|_{L^2} \|q\|_{L^2}\]
and
$C_2
  := \left( \|p\|_{H^s} + \|q\|_{H^s} \right) \|p\|_{W^{2s,4}} \|q\|_{W^{2s,4}}
      + \|p\|_{H^s}^4 \|q\|_{H^s}^4$.
\label{thm:variance_bound_appendix}
\end{theorem}

\begin{proof}
We will use the Efron-Stein inequality \citep{efronStein81} to bound the
variance of $\hat S_n$. To do this, suppose we were to draw $n$ additional IID
samples $X_1',\dots,X_n' \sim p$, and define, for all
$\ell, j \in \{1,\dots,n\}$,
\[X_j^{(\ell)}
  = \left\{
      \begin{array}{ll}
        X_j' & \mbox{ if } j = \ell \\
        X_j & \mbox{ else }
      \end{array}
    \right..\]
Let
\[\hat S_n^{(\ell)}
  := \frac{1}{n^2}
     \sum_{|z| \leq Z_n} z^{2s} \sum_{j = 1}^n \sum_{k = 1}^n
                                  \psi_z(X_j^{(\ell)}) \overline{\psi_z(Y_k)}\]
denote our estimate when we replace $X_\ell$ by $X_\ell'$. Noting the symmetry of
$\hat S_n$ in $p$ and $q$, the Efron-Stein inequality tells us that
\begin{equation}
\Var \left[ \hat S_n \right]
  \leq \sum_{\ell = 1}^n
    \E \left[
      \left|
        \hat S_n - \hat S_n^{(\ell)}
      \right|^2
    \right],
\label{ineq:efron_stein}
\end{equation}
where the expectation above (and elsewhere in this section) is taken over all
$3n$ samples $X_1,\dots,X_{2n},X_1',\dots,X_{2n}',Y_1,\dots,Y_n$. Expanding the
difference in (\ref{ineq:efron_stein}), note that any terms with $j \neq \ell$
cancel, so that
\footnote{It is useful here to note that $\overline{\psi_z(x)} = \psi_{-z}(x)$
and that $\psi_y\psi_z = \psi_{y + z}$.}
\begin{align*}
\hat S_n - \hat S_n^{(\ell)}
& = \frac{1}{n^2}
    \sum_{|z| \leq Z_n} z^{2s}
                          \sum_{j = 1}^n \sum_{k = 1}^n
  \psi_z(X_j) \overline{\psi_z(Y_k)}
  - \psi_z(X_j^{(\ell)}) \overline{\psi_z(Y_k)} \\
& = \frac{1}{n^2}
    \sum_{|z| \leq Z_n} z^{2s} (\psi_z(X_\ell) - \psi_z(X_\ell'))
                          \sum_{k = 1}^n \psi_{-z}(Y_k),
\end{align*}
and so
\begin{align}
\notag
& \left| \hat S_n - \hat S_n^{(\ell)} \right|^2 \\
& = \frac{1}{n^4}
    \sum_{|y|,|z| \leq Z_n}
        (yz)^{2s}
        (\psi_y(X_\ell) - \psi_y(X_\ell'))
        (\psi_{-z}(X_\ell) - \psi_{-z}(X_\ell'))
        \left( \sum_{k = 1}^n \psi_{-y}(Y_k) \right)
        \left( \sum_{k = 1}^n \psi_z(Y_k) \right).
\label{eq:expanded_square_of_diff}
\end{align}
Since $X_\ell$ and $X_\ell'$ are IID,
\begin{align*}
\E \left[ (\psi_y(X_\ell) - \psi_y(X_\ell'))
          (\psi_{-z}(X_\ell) - \psi_{-z}(X_\ell')) \right]
& = 2 \left( \E_{X \sim p} \left[ \psi_{y - z}(X) \right]
        - \E_{X \sim p} \left[ \psi_y(X) \right]
          \E_{X \sim p} \left[ \psi_{-z}(X) \right] \right) \\
& = 2 \left( \tilde p(y - z) - \tilde p(y) \tilde p(-z) \right),
\end{align*}
and, since $Y_1,\dots,Y_n$ are IID,
\begin{align*}
\E \left[ \left( \sum_{k = 1}^n \psi_{-y}(Y_k) \right)
            \left( \sum_{k = 1}^n \psi_z(Y_k) \right) \right]
& = n \E_{Y \sim q} \left[ \psi_{z - y}(Y) \right]
            + n(n - 1) \E_{Y \sim q} \left[ \psi_{-y}(Y) \right]
                       \E_{Y \sim q} \left[ \psi_z(Y) \right]  \\
& = n \tilde q(z - y)
  + n(n - 1) \tilde q(-y) \tilde q(z).
\end{align*}
In view of these two equalities, taking the expectation of
(\ref{eq:expanded_square_of_diff}) and using the fact that $X_\ell$ and
$X_\ell'$ are independent of $X_{n + 1},\dots,X_{2n}$,
(\ref{eq:expanded_square_of_diff}) reduces:
\begin{align}
\notag
& \E \left[ \left| \hat S_n - \hat S_n^{(\ell)} \right|^2 \right]
  = \frac{2}{n^3}
    \sum_{|y|,|z| \leq Z_n}
        (yz)^{2s}
        \left( \tilde p(y - z) - \tilde p(y) \tilde p(-z) \right)
        \left( \tilde q(z - y) + (n - 1) \tilde q(-y) \tilde q(z) \right) \\
\notag
& = \frac{2}{n^3}
    \sum_{|y|,|z| \leq Z_n}
        (yz)^{2s}
        \left(
            \tilde p(y - z) \tilde q(z - y)
            - \tilde p(y) \tilde p(-z) \tilde q(z - y)
        \right. \\
\label{ineq:intermediate_sobolev_variance_bound}
& \hspace{35mm}
        \left.
            + (n - 1) \tilde p(y - z) \tilde q(-y) \tilde q(z)
            - (n - 1) \tilde p(y) \tilde p(-z) \tilde q(-y) \tilde q(z)
        \right).
\end{align}
We now need to bound following terms in magnitude:
\begin{align}
\label{exp:variance_term_1}
& \sum_{|y|,|z| \leq Z_n} (yz)^{2s} \tilde p(y - z) \tilde q(z - y), \quad \\
\label{exp:variance_term_2}
& \sum_{|y|,|z| \leq Z_n} (yz)^{2s} \tilde p(y - z) \tilde q(-y) \tilde q(z), \\
\label{exp:variance_term_3}
\mbox{ and } \quad
& \sum_{|y|,|z| \leq Z_n} (yz)^{2s} \tilde p(y) \tilde p(-z) \tilde q(-y) \tilde q(z)
\end{align}
(the second term in (\ref{ineq:intermediate_sobolev_variance_bound}) is bounded
identically to the third term).

{\bf To bound (\ref{exp:variance_term_1}),}
we perform a change of variables, replacing $y$ by $k = y - z$:
\begin{align}
\sum_{|y|,|z| \leq Z_n} (yz)^{2s} \tilde p(y - z) \tilde q(z - y)
\label{ineq:sobolev_variance_term1_1}
& = \sum_{|k| \leq 2Z_n} \tilde p(k) \tilde q(-k)
      \sum_{j = 1}^D \sum_{z_j = \max\{-Z_n, k_j - Z_n\}}^{\min\{Z_n, k_j + Z_n\}} ((k - z)z)^{2s} \\
\label{ineq:sobolev_variance_term1_2}
& \leq \frac{2^D\Gamma(4s + 1)}{\Gamma(4s + D + 1)} Z_n^{4s + D}
      \sum_{|k| \leq 2Z_n} \tilde p(k) \tilde q(-k) \\
\label{ineq:sobolev_variance_term1_3}
& \leq C_1 Z_n^{4s + D},
\end{align}
where $C_1$ is the constant (in $n$ and $Z_n$)
\begin{equation}
C_1
  := \frac{2^D\Gamma(4s + 1)}{\Gamma(4s + D + 1)} \|p\|_2 \|q\|_2.
\label{constant:C_1_appendix}
\end{equation}
(\ref{ineq:sobolev_variance_term1_1}) and (\ref{ineq:sobolev_variance_term1_2})
follow from observing that
\[\sum_{j = 1}^D \sum_{z_j = \max\{-Z_n, k_j - Z_n\}}^{\min\{Z_n, k_j + Z_n\}}
      ((k_j - z_j)z_j)^{2s}
  = (f * f)(k_j),\]
where $f(z) := z^{2s} 1_{\{|z| \leq Z_n\}}, \forall z \in \Z^D$ and $*$ denotes
convolution (over $\Z^D$). This convolution is clearly maximized when $k = 0$,
in which case
\[(f * f)(k)
  = \sum_{|z| \leq Z_n} z^{4s}
  \leq \left( \int_{B_\infty(0, Z_n)} z^{4s} \, dz\right)
  = \frac{2^D\Gamma(4s + 1)}{\Gamma(4s + D + 1)} Z_n^{4s + D},\]
where we upper bounded the series by an integral over
\[B_\infty(0, Z_n)
  := \{z \in \R^D : \|z\|_\infty = \max\{|z_1|,...,|z_D|\} \leq Z_n\}.\]
(\ref{ineq:sobolev_variance_term1_3}) then follows via Cauchy-Schwarz.

{\bf Bounding (\ref{exp:variance_term_2})} for general $s$ is more involved and
requires rigorously defining more elaborate notions from the theory
distributions, but the basic idea is as follows:
\begin{align}
\sum_{|y|,|z| \leq Z_n} (yz)^{2s} \tilde p(y - z) \tilde q(-y) \tilde q(z)
\notag
& = \sum_{|y| \leq Z_n}
        y^{2s} \tilde q(-y)
        \sum_{|z| \leq Z_n} z^{2s} \tilde p(y - z) \tilde q(z) \\
\notag
& = \sum_{|y|,|z| \leq Z_n} y^{2s} \tilde q(-y) \tilde{\left( p_n^{(s)} q_n^{(s)} \right)}(y) \\
\notag
& \leq \sqrt{\sum_{|y| \leq Z_n} y^{2s} \left| \tilde q(y) \right|^2
          \sum_{|y| \leq Z_n} y^{2s}
                      \left( \tilde{\left( p^{(s)} q^{(s)} \right)}(y) \right)^2} \\
\label{ineq:sobolev_variance_term2}
& = \|q\|_{H^s} \|p_n^{(s)}q_n^{(s)}\|_{H^s}
  \leq \|q\|_{H^s} \|p_n\|_{W^{2s, 4}} \|q_n\|_{W^{2s, 4}}.
\end{align}
Here, $p_n^{(s)}$ and $q_n^{(s)}$ denote $s$-order fractional derivatives of
$p_n$ and $q_n$, respectively, and $W^{2s,4}$ is a Sobolev space (with
associated pseudonorm $\|\cdot\|_{W^{2s,4}}$), which can be informally thought
of as
$W^{2s, 4} := \left\{ p \in L^2 : \left( p^{(s)} \right)^2 \in H^s \right\}$.
The equality between the first and second lines follows from Theorem
\ref{thm:convolution}, and both inequalities are simply applications of
Cauchy-Schwarz. For sake of intuition, it can be noted that the above steps
are relatively elementary when $s = 0$. Now, it suffices to note that, by the
Rellich-Kondrachov embedding theorem
\citep{rellich30sobolevEmbedding,evans10PDEs},
$W^{2s,4} \subseteq H^{s'}$, and hence
$\|p_n\|_{W^{2s, 4}} \leq \|p\|_{W^{2s,4}} < \infty$, as long as
$s' \geq 2s + \frac{D}{4}$.

{\bf Bounding (\ref{exp:variance_term_3})} is a simple application of
Cauchy-Schwarz:
\begin{align}
\sum_{|y|,|z| \leq Z_n} (yz)^{2s} \tilde p(y) \tilde p(-z) \tilde q(-y) \tilde q(z)
\notag
& = \left( \sum_{|y| \leq Z_n} y^{2s} \tilde p(y) \tilde q(-y) \right)
    \left( \sum_{|z| \leq Z_n} z^{2s} \tilde p(-z) \tilde q(z) \right) \\
\notag
& \leq \left( \sum_{|z| \leq Z_n} z^{2s} \left| \tilde p(z) \right|^2 \right)^2
       \left( \sum_{|z| \leq Z_n} z^{2s} \left| \tilde q(z) \right|^2 \right)^2 \\
\label{ineq:sobolev_variance_term3}
& = \|p\|_{H^s}^4 \|q\|_{H^s}^4
\end{align}

Plugging (\ref{ineq:sobolev_variance_term1_3}),
(\ref{ineq:sobolev_variance_term2}), and (\ref{ineq:sobolev_variance_term3})
into
(\ref{ineq:intermediate_sobolev_variance_bound}) gives
\[\E \left[ \left| \hat S_n - \hat S_n^{(\ell)} \right|^2 \right]
  \leq 2C_1 \frac{Z_n^{4s + D}}{n^3}
  + \frac{C_2}{n^2},\]
where $C_2$ denotes the constant (in $n$ and $Z_n$)
\begin{equation}
C_2
  := \left( \|p\|_{H^s} + \|q\|_{H^s} \right) \|p\|_{W^{2s,4}} \|q\|_{W^{2s,4}}
      + \|p\|_{H^s}^4 \|q\|_{H^s}^4.
\label{constant:C_2}
\end{equation}
Plugging this into the Efron-Stein inequality (\ref{ineq:efron_stein}) gives,
by symmetry of $\hat S_n$ in $X_1,...,X_n$,
\[\Var \left[ \hat S_n \right]
  \leq 2C_1 \frac{Z_n^{4s + D}}{n^2} + \frac{C_2}{n}.\]
\end{proof}

\section{Proofs of Asymptotic Distributions}
\label{sec:asymptotic_results_appendix}
\begin{theorem}
Suppose that, for some $s' > 2s + \frac{D}{4}$, $p, q \in H^{s'}$, and
suppose $Z_n n^{\frac{1}{4(s - s')}} \to \infty$ and
$Z_n n^{-\frac{1}{4s + D}} \to 0$ as $n \to \infty$. Then, $\hat S_n$ is
asymptotically normal with mean $\langle p, q \rangle$. In particular, for
$j \in \{1, \dots, n\}$, define the following quantities:
\[W_j
  := \begin{bmatrix}
    Z_n^s e^{i Z_n X_j} \\
    \vdots \\
    e^{i X_j} \\
    e^{i X_j} \\
    \vdots \\
    Z_n^s e^{-i Z_n X_j}
  \end{bmatrix},
    \quad
  V_j
  := \begin{bmatrix}
    Z_n^s e^{i Z_n Y_j} \\
    \vdots \\
    e^{i Y_j} \\
    e^{i Y_j} \\
    \vdots \\
    Z_n^s e^{-i Z_n Y_j}
  \end{bmatrix},
    \quad
  \overline W := \frac{1}{n} \sum_{j = 1}^n W_j,
    \quad
  \overline V := \frac{1}{n} \sum_{j = 1}^n V_j
  \in \R^{2Z_n},\]
\[\Sigma_W := \frac{1}{n} \sum_{j = 1}^n (W_j - \overline W) (W_j - \overline W)^T,
  \quad \mbox{ and } \quad
  \Sigma_V := \frac{1}{n} \sum_{j = 1}^n (V_j - \overline V) (V_j - \overline V)^T
  \in \R^{2Z_n \times 2Z_n}.\]
Then, for
\[\hat \sigma_n^2
  := \begin{bmatrix}
      \overline V \\
      \overline W
    \end{bmatrix}^T
    \begin{bmatrix}
      \Sigma_W & 0 \\
      0 & \Sigma_V
    \end{bmatrix}
    \begin{bmatrix}
      \overline V \\
      \overline W
    \end{bmatrix},\]
we have
\[\sqrt{n} \left( \frac{\hat S_n - \langle p, q \rangle_{H^s}}{\hat\sigma_n} \right)
  \stackrel{D}{\to} \mathcal{N}(0, 1).\]
\label{thm:asymp_normal_appendix}
\end{theorem}
\begin{proof}
By the bias bound and the assumption $Z_n^{4(s - s')} n \to \infty$,
it suffices to show
\begin{equation}
\sqrt{n} \left( \frac{\hat S_n - \E \left[ \hat S_n \right]}{\sigma_n} \right)
  \stackrel{D}{\to} \mathcal{N}(0, 1).
  \quad \mbox{ as } \quad
  n \to \infty.
\label{limit:unbiased_normal}
\end{equation}
Let
\[\tilde p_{Z_n} :=
  \begin{bmatrix}
    \tilde p(-Z_n) \\
    \tilde p(-Z_n + 1) \\
    \vdots \\
    \tilde p(Z_n - 1) \\
    \tilde p(Z_n) \\
  \end{bmatrix},
  \quad
  \hat p_{Z_n} :=
  \begin{bmatrix}
    \hat p(-Z_n) \\
    \hat p(-Z_n + 1) \\
    \vdots \\
    \hat p(Z_n - 1) \\
    \hat p(Z_n) \\
  \end{bmatrix},\]
\[\tilde q_{Z_n} :=
  \begin{bmatrix}
    \tilde q(-Z_n) \\
    \tilde q(-Z_n + 1) \\
    \vdots \\
    \tilde q(Z_n - 1) \\
    \tilde q(Z_n) \\
  \end{bmatrix},
  \quad \mbox{ and } \quad
  \hat q :=
  \begin{bmatrix}
    \hat q(-Z_n) \\
    \hat q(-Z_n + 1) \\
    \vdots \\
    \hat q(Z_n - 1) \\
    \hat q(Z_n) \\
  \end{bmatrix}.\]
Since $\hat p_{Z_n}$ and $\hat q_{Z_n}$ are empirical means of bounded random
vectors with means $\tilde p_{Z_n}$ and $\tilde q_{Z_n}$, respectively, by the
central limit theorem, as $n \to \infty$,
\[\sqrt{n} \left( \hat p_{Z_n} - \tilde p_{Z_n} \right) \stackrel{D}{\to}
\mathcal{N}(0, \Sigma_p)
  \quad \mbox{ and } \quad
  \sqrt{n} \left( \hat q_{Z_n} - \tilde q_{Z_n} \right) \stackrel{D}{\to}
\mathcal{N}(0, \Sigma_q),\]
where
\[\left( \Sigma_p \right)_{w,z}
      := \Cov_{X \sim p} \left( \psi_w(X), \psi_z(X) \right)
  \quad \mbox{ and } \quad
  \left( \Sigma_q \right)_{w,z}
      := \Cov_{X \sim q} \left( \psi_w(X), \psi_z(X) \right).\]
Define $h : \R^{2Z_n + 1} \times \R^{2Z_n + 1} \to \R$ by
$h(x,y) = \sum_{z = -Z_n}^{Z_n} z^{2s} x_z y_{-z}$, and note that
\[\sigma_n^2
  := \left( \nabla h \left( \tilde p_{Z_n}, \tilde q_{Z_n} \right) \right)'
    \begin{bmatrix}
      \Sigma_p & 0 \\
      0 & \Sigma_q
    \end{bmatrix}
    \left( \nabla h \left( \tilde p_{Z_n}, \tilde q_{Z_n} \right) \right).\]
\eqref{limit:unbiased_normal} follows by the delta method.
\end{proof}

\begin{theorem}
Suppose that, for some $s' > 2s + \frac{D}{4}$, $p, q \in H^{s'}$, and
suppose $Z_n n^{\frac{1}{4(s - s')}} \to \infty$ and
$Z_n n^{-\frac{1}{4s + D}} \to 0$ as $n \to \infty$.
For $j \in \{1,\dots,n\}$, define
\[W_j :=
  \begin{bmatrix}
    Z_n^s \left( e^{iZ_n X_j} - e^{iZ_n Y_j} \right) \\
    \vdots \\
    e^{iX_j} - e^{iY_j} \\
    e^{-iX_j} - e^{-iY_j} \\
    \vdots \\
    Z_n^s \left( e^{-iZ_n X_j} - e^{-iZ_n Y_j} \right)
  \end{bmatrix}
  \in \R^{2Z_n}.\]
Let
\[\overline W := \frac{1}{n} \sum_{j = 1}^n W_j
  \quad \mbox{ and } \quad
  \Sigma
  := \frac{1}{n} \sum_{j = 1}^n \left( W_j - \overline W \right)
                                \left( W_j - \overline W \right)^T\]
denote the empirical mean and covariance of $W$, and define
$T := n \overline W^T \Sigma\inv \overline W$. Then, if $p = q$, then
\[Q_{\chi^2(2Z_n)}(T) \stackrel{D}{\to} \operatorname{Uniform}([0, 1])
  \quad \mbox{ as } \quad
  n \to \infty,\]
where $Q_{\chi^2(2Z_n)} : [0, \infty) \to [0, 1]$ denotes the quantile function
(inverse CDF) of the $\chi^2$ distribution $\chi^2(2Z_n)$ with $2Z_n$ degrees of
freedom.
\label{thm:asymptotic_test_distribution_appendix}
\end{theorem}
\begin{proof}
Since, as shown in the proof of the previous theorem, the distance estimate is a
sum of squared asymptotically normal, zero-mean random variables, this is a
standard result in multivariate statistics. See, for example, Theorem 5.2.3 of
\citet{anderson03multivariateStats}.
\end{proof}

\section{Generalizations: Weak and Fractional Derivatives}
\label{sec:fractional_derivatives}

As mentioned in the main text, our estimator and analysis can be generalized
nicely to non-integer $s$ using an appropriate notion of fractional derivative.

For non-negative integers $s$, let $\delta^{(s)}$ denote the measure underlying
of the $s$-order derivative operator at $0$; that is, $\delta^{(s)}$ is the
distribution such that
\[\int_\R f(x) \delta^{(s)}(x) \, dx = f^{(s)}(0),\]
for all test functions $f \in H^s$. Then, for all $z \in \R$, the Fourier
transform of $\delta^{(s)}$ is
\[\tilde \delta(z)
  = \int_\R e^{-izx} \delta^{(s)}(x) \, dx
  = (-iz)^s.\]
Thus, we can naturally generalize the derivative operator $\delta^{(s)}$ to
general $s \in [0, \infty)$ as the inverse Fourier transform of the function
$z \mapsto (-iz)^s$. Generalization to differentiation at an arbitrary
$y \in \R$ follows from translation properties of the Fourier transform, and,
in multiple dimensions, for $s \in \R^D$, we can consider the inverse Fourier
transform of $z \in \R^D \mapsto \prod_{j = 1}^D (iz_j)^{s_j}$.

With this definition in place, we can prove the following the Convolution
Theorem, which equates a particular weighted convolution of Fourier transforms
and a product of particular fractional derivatives. Note that we will only need
this result in the case that $f$ is a trigonometric polynomial (i.e., $\tilde f$
has finite support), because we apply it only to $p_n$ and $q_n$. Hence, the sum
below has only finitely many non-zero terms and commutes freely with integrals.
\begin{theorem}
Suppose $p, q \in L^2$ are trigonometric polynomials. Then,
$\forall s \in [0, \infty)$, and $y \in \Z^D$,
\[\sum_{z \in \Z^D} z^{2s} \tilde p(y - z) \tilde q(z)
  = \tilde{\left( p^{(s)} q^{(s)} \right)}(y).\]
\label{thm:convolution}
\end{theorem}
\begin{proof}
By linearity of the integral,
\begin{align*}
\sum_{z \in \Z^D} z^{2s} \tilde p(y - z) \tilde q(z)
& = \sum_{z \in \Z^D} z^{2s}
        \int_{\R^D} p(x_1) e^{-i\langle y - z, x_1 \rangle} \, dx_1
        \int_{\R^D} q(x_2) e^{-i\langle z, x_2 \rangle} \, dx_2 \\
& = \int_{\R^D} \int_{\R^D} p(x_1) q(x_2) e^{-i\langle y, x_1 \rangle}
  \sum_{z \in \Z^D} z^{2s} e^{i\langle z, x_1 - x_2 \rangle} \, dx_1 \, dx_2 \\
& = \int_{\R^D} \int_{\R^D} p(x_1) q(x_2) e^{-i\langle y, x_1 \rangle}
  \delta^{(s)}(x_1 - x_2) \, dx_1 \, dx_2 \\
& = \int_{\R^D} p^{(s)}(x) q^{(s)}(x) e^{-i\langle y, x \rangle} \, dx
  = \tilde{\left( p^{(s)} q^{(s)} \right)}(y).
\end{align*}
\end{proof}

\subsubsection*{Acknowledgments}
This material is based upon work supported by a National Science Foundation
Graduate Research Fellowship to the first author under Grant No. DGE-1252522.

{
\small
\bibliography{biblio}

\begin{thebibliography}{43}
\providecommand{\natexlab}[1]{#1}
\providecommand{\url}[1]{\texttt{#1}}
\expandafter\ifx\csname urlstyle\endcsname\relax
  \providecommand{\doi}[1]{doi: #1}\else
  \providecommand{\doi}{doi: \begingroup \urlstyle{rm}\Url}\fi

\bibitem[Anderson et~al.(1994)Anderson, Hall, and
  Titterington]{anderson94L_2TwoSampleTest}
Niall~H Anderson, Peter Hall, and D~Michael Titterington.
\newblock Two-sample test statistics for measuring discrepancies between two
  multivariate probability density functions using kernel-based density
  estimates.
\newblock \emph{Journal of Multivariate Analysis}, 50\penalty0 (1):\penalty0
  41--54, 1994.

\bibitem[Anderson(2003)]{anderson03multivariateStats}
TW~Anderson.
\newblock An introduction to multivariate statistical analysis.
\newblock \emph{Wiley}, 2003.

\bibitem[Bickel and Ritov(1988)]{bickel88squaredDensityDerivatives}
Peter~J Bickel and Ya’acov Ritov.
\newblock Estimating integrated squared density derivatives: sharp best order
  of convergence estimates.
\newblock \emph{Sankhy{\=a}: The Indian Journal of Statistics, Series A}, pages
  381--393, 1988.

\bibitem[Birg{\'e} and Massart(1995)]{birge95functionalLowerBounds}
Lucien Birg{\'e} and Pascal Massart.
\newblock Estimation of integral functionals of a density.
\newblock \emph{The Annals of Statistics}, pages 11--29, 1995.

\bibitem[Chwialkowski et~al.(2015)Chwialkowski, Ramdas, Sejdinovic, and
  Gretton]{chwialkowski15fastTwoSample}
Kacper~P Chwialkowski, Aaditya Ramdas, Dino Sejdinovic, and Arthur Gretton.
\newblock Fast two-sample testing with analytic representations of probability
  measures.
\newblock In \emph{Advances in Neural Information Processing Systems}, pages
  1972--1980, 2015.

\bibitem[Efron and Stein(1981)]{efronStein81}
Bradley Efron and Charles Stein.
\newblock The jackknife estimate of variance.
\newblock \emph{The Annals of Statistics}, pages 586--596, 1981.

\bibitem[Epps and Singleton(1986)]{epps86twoSampleECF}
TW~Epps and Kenneth~J Singleton.
\newblock An omnibus test for the two-sample problem using the empirical
  characteristic function.
\newblock \emph{Journal of Statistical Computation and Simulation}, 26\penalty0
  (3-4):\penalty0 177--203, 1986.

\bibitem[Evans(2010)]{evans10PDEs}
Lawrence~C Evans.
\newblock \emph{Partial differential equations}.
\newblock American Mathematical Society, 2010.

\bibitem[Gin{\'e} and Nickl(2008)]{gine08L2Norm}
Evarist Gin{\'e} and Richard Nickl.
\newblock A simple adaptive estimator of the integrated square of a density.
\newblock \emph{Bernoulli}, pages 47--61, 2008.

\bibitem[Goria et~al.(2005)Goria, Leonenko, Mergel, and Inverardi]{goria05new}
M.~N. Goria, N.~N. Leonenko, V.~V. Mergel, and P.~L.~Novi Inverardi.
\newblock A new class of random vector entropy estimators and its applications
  in testing statistical hypotheses.
\newblock \emph{J. Nonparametric Statistics}, 17:\penalty0 277--297, 2005.

\bibitem[Gretton et~al.(2006)Gretton, Borgwardt, Rasch, Sch{\"o}lkopf, and
  Smola]{gretton06MMD_NIPS}
Arthur Gretton, Karsten~M Borgwardt, Malte Rasch, Bernhard Sch{\"o}lkopf, and
  Alex~J Smola.
\newblock A kernel method for the two-sample-problem.
\newblock In \emph{Advances in neural information processing systems}, pages
  513--520, 2006.

\bibitem[Gretton et~al.(2012)Gretton, Borgwardt, Rasch, Sch{\"o}lkopf, and
  Smola]{gretton12MMD_JMLR}
Arthur Gretton, Karsten~M Borgwardt, Malte~J Rasch, Bernhard Sch{\"o}lkopf, and
  Alexander Smola.
\newblock A kernel two-sample test.
\newblock \emph{The Journal of Machine Learning Research}, 13\penalty0
  (1):\penalty0 723--773, 2012.

\bibitem[Hall and Marron(1987)]{hall87SobolevNorms}
Peter Hall and James~Stephen Marron.
\newblock Estimation of integrated squared density derivatives.
\newblock \emph{Statistics \& Probability Letters}, 6\penalty0 (2):\penalty0
  109--115, 1987.

\bibitem[Heathcote(1972)]{heathcote1972test}
CE~Heathcote.
\newblock A test of goodness of fit for symmetric random variables.
\newblock \emph{Australian Journal of Statistics}, 14\penalty0 (2):\penalty0
  172--181, 1972.

\bibitem[Hero et~al.(2002)Hero, Ma, Michel, and Gorman]{hero2002aes}
A.~O. Hero, B.~Ma, O.~J.~J. Michel, and J.~Gorman.
\newblock Applications of entropic spanning graphs.
\newblock \emph{IEEE Signal Processing Magazine}, 19\penalty0 (5):\penalty0
  85--95, 2002.

\bibitem[Ibragimov and Khasminskii(1978)]{ibragimov78nonparametricFunctionals}
IA~Ibragimov and RZ~Khasminskii.
\newblock On the nonparametric estimation of functionals.
\newblock In \emph{Symposium in Asymptotic Statistics}, pages 41--52, 1978.

\bibitem[Kandasamy et~al.(2015)Kandasamy, Krishnamurthy, Poczos, Wasserman,
  et~al.]{kandasamy15vonMises}
Kirthevasan Kandasamy, Akshay Krishnamurthy, Barnabas Poczos, Larry Wasserman,
  et~al.
\newblock Nonparametric von mises estimators for entropies, divergences and
  mutual informations.
\newblock In \emph{Advances in Neural Information Processing Systems}, pages
  397--405, 2015.

\bibitem[Kreiss and Oliger(1979)]{kreiss79Fourierstability}
Heinz-Otto Kreiss and Joseph Oliger.
\newblock Stability of the {F}ourier method.
\newblock \emph{SIAM Journal on Numerical Analysis}, 16\penalty0 (3):\penalty0
  421--433, 1979.

\bibitem[Krishnamurthy et~al.(2014{\natexlab{a}})Krishnamurthy, Kandasamy,
  Poczos, and Wasserman]{krishnamurthy2014L2Divergence}
Akshay Krishnamurthy, Kirthevasan Kandasamy, Barnabas Poczos, and Larry
  Wasserman.
\newblock On estimating ${L}_2^{2}$ divergence.
\newblock \emph{arXiv preprint arXiv:1410.8372}, 2014{\natexlab{a}}.

\bibitem[Krishnamurthy et~al.(2014{\natexlab{b}})Krishnamurthy, Kandasamy,
  Poczos, and Wasserman]{krishnamurthy2014renyiAndFriends}
Akshay Krishnamurthy, Kirthevasan Kandasamy, Barnabas Poczos, and Larry
  Wasserman.
\newblock Nonparametric estimation of renyi divergence and friends.
\newblock \emph{arXiv preprint arXiv:1402.2966}, 2014{\natexlab{b}}.

\bibitem[Laurent(1992)]{laurent92integralFunctionals}
B{\'e}atrice Laurent.
\newblock \emph{Efficient estimation of integral functionals of a density}.
\newblock Universit{\'e} de Paris-sud, D{\'e}partement de math{\'e}matiques,
  1992.

\bibitem[Laurent et~al.(1996)]{laurent96integralFunctionals}
B{\'e}atrice Laurent et~al.
\newblock Efficient estimation of integral functionals of a density.
\newblock \emph{The Annals of Statistics}, 24\penalty0 (2):\penalty0 659--681,
  1996.

\bibitem[Leonenko et~al.(2008)Leonenko, Pronzato, Savani,
  et~al.]{leonenko08RenyiEntropy}
Nikolai Leonenko, Luc Pronzato, Vippal Savani, et~al.
\newblock A class of r{\'e}nyi information estimators for multidimensional
  densities.
\newblock \emph{The Annals of Statistics}, 36\penalty0 (5):\penalty0
  2153--2182, 2008.

\bibitem[Leoni(2009)]{leoni09Sobolev}
Giovanni Leoni.
\newblock \emph{A first course in Sobolev spaces}, volume 105.
\newblock American Mathematical Society Providence, RI, 2009.

\bibitem[Moon and Hero(2014{\natexlab{a}})]{moon14divergencesConfidence}
Kevin Moon and Alfred Hero.
\newblock Multivariate f-divergence estimation with confidence.
\newblock In \emph{Advances in Neural Information Processing Systems}, pages
  2420--2428, 2014{\natexlab{a}}.

\bibitem[Moon and Hero(2014{\natexlab{b}})]{moon14divergencesEnsemble}
Kevin~R Moon and Alfred~O Hero.
\newblock Ensemble estimation of multivariate f-divergence.
\newblock In \emph{Information Theory (ISIT), 2014 IEEE International Symposium
  on}, pages 356--360. IEEE, 2014{\natexlab{b}}.

\bibitem[Moon et~al.(2016)Moon, Sricharan, Greenewald, and
  Hero~III]{moon16improvingConvergence}
Kevin~R Moon, Kumar Sricharan, Kristjan Greenewald, and Alfred~O Hero~III.
\newblock Improving convergence of divergence functional ensemble estimators.
\newblock \emph{arXiv preprint arXiv:1601.06884}, 2016.

\bibitem[Pardo(2005)]{pardo05divergenceMeasures}
Leandro Pardo.
\newblock \emph{Statistical inference based on divergence measures}.
\newblock CRC Press, 2005.

\bibitem[P{\'o}czos and Schneider(2011)]{poczos11alphaDivergence}
Barnab{\'a}s P{\'o}czos and Jeff~G Schneider.
\newblock On the estimation of alpha-divergences.
\newblock In \emph{International Conference on Artificial Intelligence and
  Statistics}, pages 609--617, 2011.

\bibitem[P{\'o}czos et~al.(2012{\natexlab{a}})P{\'o}czos, Xiong, and
  Schneider]{poczos12divergenceApplication}
Barnab{\'a}s P{\'o}czos, Liang Xiong, and Jeff Schneider.
\newblock Nonparametric divergence estimation with applications to machine
  learning on distributions.
\newblock \emph{arXiv preprint arXiv:1202.3758}, 2012{\natexlab{a}}.

\bibitem[P{\'o}czos et~al.(2012{\natexlab{b}})P{\'o}czos, Xiong, Sutherland,
  and Schneider]{poczos12kernelsForImages}
Barnab{\'a}s P{\'o}czos, Liang Xiong, Dougal~J Sutherland, and Jeff Schneider.
\newblock Nonparametric kernel estimators for image classification.
\newblock In \emph{Computer Vision and Pattern Recognition (CVPR), 2012 IEEE
  Conference on}, pages 2989--2996. IEEE, 2012{\natexlab{b}}.

\bibitem[Principe(2010)]{principe10RenyiEntropy}
Jose~C Principe.
\newblock \emph{Information theoretic learning: Renyi's entropy and kernel
  perspectives}.
\newblock Springer Science \& Business Media, 2010.

\bibitem[Quadrianto et~al.(2009)Quadrianto, Petterson, and
  Smola]{quadrianto09distributionMatching}
Novi Quadrianto, James Petterson, and Alex~J Smola.
\newblock Distribution matching for transduction.
\newblock In \emph{Advances in Neural Information Processing Systems}, pages
  1500--1508, 2009.

\bibitem[Ram et~al.(2009)Ram, Lee, March, and Gray]{ram09linearTime}
Parikshit Ram, Dongryeol Lee, William March, and Alexander~G Gray.
\newblock Linear-time algorithms for pairwise statistical problems.
\newblock In \emph{Advances in Neural Information Processing Systems}, pages
  1527--1535, 2009.

\bibitem[Rellich(1930)]{rellich30sobolevEmbedding}
Franz Rellich.
\newblock Ein satz {\"u}ber mittlere konvergenz.
\newblock \emph{Nachrichten von der Gesellschaft der Wissenschaften zu
  G{\"o}ttingen, Mathematisch-Physikalische Klasse}, 1930:\penalty0 30--35,
  1930.

\bibitem[Schweder(1975)]{schweder75windowVariance}
Tore Schweder.
\newblock Window estimation of the asymptotic variance of rank estimators of
  location.
\newblock \emph{Scandinavian Journal of Statistics}, pages 113--126, 1975.

\bibitem[Singh and P{\'o}czos(2014{\natexlab{a}})]{singh14RenyiDivergence}
Shashank Singh and Barnab{\'a}s P{\'o}czos.
\newblock Generalized exponential concentration inequality for renyi divergence
  estimation.
\newblock In \emph{Proceedings of The 31st International Conference on Machine
  Learning}, pages 333--341, 2014{\natexlab{a}}.

\bibitem[Singh and P{\'o}czos(2014{\natexlab{b}})]{singh14densityFunctionals}
Shashank Singh and Barnab{\'a}s P{\'o}czos.
\newblock Exponential concentration of a density functional estimator.
\newblock In \emph{Advances in Neural Information Processing Systems}, pages
  3032--3040, 2014{\natexlab{b}}.

\bibitem[Tsybakov(2008)]{Tsybakov:2008:INE:1522486}
A.B. Tsybakov.
\newblock \emph{Introduction to Nonparametric Estimation}.
\newblock Springer Publishing Company, Incorporated, 1st edition, 2008.
\newblock ISBN 0387790519, 9780387790510.

\bibitem[Wolsztynski et~al.(2005)Wolsztynski, Thierry, and
  Pronzato]{Wolsztynski85minimum}
E.~Wolsztynski, E.~Thierry, and L.~Pronzato.
\newblock Minimum-entropy estimation in semi-parametric models.
\newblock \emph{Signal Process.}, 85\penalty0 (5):\penalty0 937--949, 2005.
\newblock ISSN 0165-1684.
\newblock \doi{http://dx.doi.org/10.1016/j.sigpro.2004.11.028}.

\bibitem[Zaremba et~al.(2013)Zaremba, Gretton, and Blaschko]{zaremba13blockMMD}
Wojciech Zaremba, Arthur Gretton, and Matthew Blaschko.
\newblock B-test: A non-parametric, low variance kernel two-sample test.
\newblock In \emph{Advances in neural information processing systems}, pages
  755--763, 2013.

\bibitem[Zhao and Meng(2015)]{zhao15fastMMD}
Ji~Zhao and Deyu Meng.
\newblock Fastmmd: Ensemble of circular discrepancy for efficient two-sample
  test.
\newblock \emph{Neural computation}, 27\penalty0 (6):\penalty0 1345--1372,
  2015.

\bibitem[Zygmund(2002)]{zygmund02trigSeries}
Antoni Zygmund.
\newblock \emph{Trigonometric series}, volume~1.
\newblock Cambridge university press, 2002.

\end{thebibliography}
\bibliographystyle{plainnat}
}

\end{document}